\documentclass[reqno]{amsart}

\usepackage{mathrsfs}
\usepackage{amssymb}
\usepackage{amsfonts}
\usepackage{amsmath}
\usepackage{txfonts}
\usepackage{dsfont}
\usepackage{amscd}
\usepackage[colorlinks, linkcolor=blue, anchorcolor=blue, citecolor=blue]{hyperref}

\title{Structure of non-solvable cyclic metric Lie algebras}
\author{Huihui An, Ju Tan \and Zaili Yan$^{*}$}
\address[Huihui An]{ School of Mathematics, Liaoning Normal University, Dalian, Liaoning Province, 116029, People's Republic of China}
\email[]{anhh@lnnu.edu.cn }
\address[Ju Tan]{School of Microelectronics and Data Science, Anhui Provincial Joint Key Laboratory of Disciplines for Industrial Big Data Analysis and Intelligent Decision, Anhui University of Technology, Maanshan, 243032, People's Republic of China}
\email[]{tanju2007@163.com}
\address[Zaili Yan]{School of Mathematics and Statistics, Ningbo University, Ningbo, Zhejiang Province, 315211,  People's Republic of China}
\email[]{yanzaili@nbu.edu.cn}
\thanks{H. An is supported by special fund for basic scientific research expenses of universities in Liaoning Province (LJ212410165012). J. Tan  is  supported  by Excellent Young Teacher Cultivation Project of Anhui Province (no.YQYB2024018), Provincial quality project (Graduate Education) course ideological and political demonstration course (no. 2024szsfkc057).
 $^{*}$Z. Yan is the corresponding author and is supported by Zhejiang Provincial Natural Science Foundation of China under Grant No. LMS25A010010.}
\date{}

\newtheorem{thm}{Theorem}[section]
\newtheorem{prop}[thm]{Proposition}
\newtheorem{lem}[thm]{Lemma}
\newtheorem{cor}[thm]{Corollary}
\theoremstyle{definition}
\newtheorem{defn}[thm]{Definition}

\newtheorem{example}[thm]{Example}

\newtheorem{rem}[thm]{Remark}

\begin{document}

\begin{abstract}
This paper presents a systematic study of the structure of non-solvable cyclic metric Lie algebras.
 A cyclic metric is a symmetric bilinear form satisfying a cyclic cocycle condition, which arises naturally in the contexts of  non-associative algebras  and homogeneous pseudo-Riemannian manifolds. Firstly,
we drive some sufficient conditions for a cyclic metric Lie algebra to be an orthogonal direct product of its semisimple and solvable parts. Then we introduce the notion of cyclic quadruples to analyze the interaction between semisimple and radical components. Finally, we use the double extension method to provide a complete characterization of non-degenerate cyclic metric Lie algebras that are neither semisimple nor solvable, over the fields of complex and real numbers.

\medskip
\textbf{Mathematics Subject Classification 2020}:   17B05, 53C30,  17B20.

\medskip
\textbf{Key words}: Cyclic metric Lie algebras; homogeneous structures;   double extensions.

\end{abstract}
\maketitle
\section{Introduction and main results}
A cyclic metric $\mathbf{B}$ on a Lie algebra $\mathfrak{g}$ over a field $\mathbb{F}$ is a symmetric bilinear form $\mathbf{B}: \mathfrak{g}\times \mathfrak{g}\rightarrow \mathbb{F}$ satisfying the identity
 \begin{equation*}
 \mathbf{B}([x,y],z)+\mathbf{B}([y,z],x)+\mathbf{B}([z,x],y)=0, \quad \forall x,y,z\in\mathfrak{g}.
 \end{equation*}
Such forms are also referred to as "commutative 2-cocycles" \cite{Dzhum10}  and arise naturally in several contexts: the study of non-associative algebras satisfying certain skew-symmetric identities \cite{Dzhum09}, the description of the second cohomology of current Lie algebras \cite{zusman94},  the structure theory of two-sided Alia algebras \cite{DB09} and anti-pre-Lie algebras \cite{LB22,LB24}.
For instance, it is shown in \cite[Theorem 2.19]{LB22} that
if $\mathfrak{g}$  admits a non-degenerate cyclic metric,  then there exists a compatible anti-pre-Lie algebra structure on $\mathfrak{g}$.
A complete classification of cyclic metrics on classical Lie algebras was given by Dzhumadil'daev and Bakirova in \cite{DB09},  who proved the following fundamental result.
\begin{thm}[\cite{DB09}]\label{1-thm-main1}
Except for the simple Lie algebra $\mathfrak{sl}(2,\mathbb{C})$,
no complex simple Lie algebra admits a non-zero cyclic metric.
\end{thm}

Cyclic metrics also play a central role in the study of homogeneous structures on Riemannian and pseudo-Riemannian manifolds \cite{cl19,tv83}. Let $(M,g)$ be a connected Riemannian or pseudo-Riemannian manifold. A homogeneous structure  on $(M,g)$ is a tensor field $S$ of type $(1,2)$ such that
$$\widetilde{\nabla}g=\widetilde{\nabla}R=\widetilde{\nabla}S=0,$$
where $\widetilde{\nabla}=\nabla-S$, $\nabla$ is the Levi-Civita connection of $(M,g)$ and $R$ is the Riemannian curvature tensor of $(M,g)$.
Ambrose and Singer in \cite{as58} proved that,  a connected, simply connected and complete Riemannian manifold is homogeneous (i.e., the isometry group $\mathrm{Iso}(M,g)$ acts transitively on $M$) if and only if it admits a homogeneous structure.
This result was later extended to the pseudo-Riemannian setting by Gadea and Oubi\~{n}a \cite{go97}.

To understand the algebraic underpinnings of these geometric structures, consider a real vector space $V$
equipped with a non-degenerate inner product $\langle\cdot,\cdot\rangle$,
which serves as a model space for the tangent space at any point of a homogeneous pseudo-Riemannian manifold $(M,g)$. Define the space of algebraic $(0,3)$-tensors
 $$\mathcal{S}(V)=\{S\in \otimes^{3}V^{*}|S(x,y,z)=-S(x,z,y), \forall x,y,z\in V\},$$
 which encapsulates the same skew-symmetry condition as $\widetilde{\nabla}g=0$  (where $S(x,y,z)=g(S_x{y},z)$).
 The inner product on $V$ induces a natural inner product on  $\mathcal{S}(V)$, leading to an orthogonal decomposition \cite{go97,tv83}
$$\mathcal{S}(V)=\mathcal{S}_{1}(V)\oplus \mathcal{S}_{2}(V)\oplus\mathcal{S}_{3}(V),$$
where  the components are characterized by
$$\mathcal{S}_{1}(V)\oplus \mathcal{S}_{2}(V)
=\{S\in\mathcal{S}(V)|S(x,y,z)+S(y,z,x)+S(z,x,y)=0, \forall x,y,z\in V \},$$
$$\mathcal{S}_{3}(V)=\{S\in\mathcal{S}(V)|S(x,y,z)+S(y,x,z)=0,\forall x,y,z\in V\}.$$
This decomposition is geometrically significant: homogeneous pseudo-Riemannian manifolds admitting a homogeneous structure $S\in \mathcal{S}_{3}$ are precisely the naturally reductive spaces \cite{go97},
while those with $S\in \mathcal{S}_{1}\oplus \mathcal{S}_{2}$ are termed cotorsionless or cyclic homogeneous manifolds. The latter class has been extensively studied in \cite{bie97,deleo14,deleoM08,fp06,ggo15,ggo16,kt87,pv91}.

A particularly tractable setting arises when the manifold is itself a Lie group.
Let $(M=G,g)$ be a connected Lie group $G$ equipped  with a left invariant pseudo-Riemannian metric $g$, uniquely determined at the Lie algebraic level by  a non-degenerate inner product $\langle\cdot,\cdot\rangle$ on the Lie algebra $\mathfrak{g}$.
The  tensor $\widetilde{S}$ defined by $\widetilde{S}_{x}y=\nabla_{x}y$ for $x,y\in \mathfrak{g}$ yields a canonical homogeneous structure on  $(M=G,g)$.
A fundamental result states that $g$ is also right invariant if and only if
$\langle\cdot,\cdot\rangle$ is ad-invariant,  which occurs precisely when $\widetilde{S}\in \mathcal{S}_{3}$.
The structure and classification of ad-invariant metric Lie algebras have been deeply investigated in \cite{fs87,mr85}, see \cite{ovando16} for a survey.
 Conversely, if  $\widetilde{S}\in \mathcal{S}_{1}\oplus \mathcal{S}_{2}$, the metric is called cyclic, and the corresponding inner product $\langle\cdot,\cdot\rangle$ is a cyclic metric.
The geometric properties and classification of such cyclic metric Lie groups are explored in \cite{abol25,cl16,ggo15,tz21}.

In this paper, we undertake a systematic study of the structure of cyclic metric Lie algebras, with a focus on the non-solvable case. Our investigation leads to several key results, which we now summarize.
\begin{thm}\label{1-thm-main1.5}
Let  $(\mathfrak{g},\mathbf{B})$ be a  cyclic metric Lie algebra over $\mathbb{F}$ with   a Levi-decomposition $\mathfrak{g}=\mathfrak{s}+\mathfrak{r}$, where $\mathfrak{s}$ denotes a semisimple Levi-factor of $\mathfrak{g}$ and $\mathfrak{r}$ is the radical of $\mathfrak{g}$.
Suppose that $\mathbf{B}|_{\mathfrak{r}}$ is non-degenerate,
or $\mathbf{B}|_{\mathrm{nil}\,\mathfrak{g}}$  is non-degenerate or Lorentz, then $(\mathfrak{g},\mathbf{B})=(\mathfrak{s},\mathbf{B}|_{\mathfrak{s}})\oplus(\mathfrak{r},\mathbf{B}|_{\mathfrak{r}})$
  is the orthogonal direct product of a semisimple cyclic metric Lie algebra
$(\mathfrak{s},\mathbf{B}|_{\mathfrak{s}})$
 and a solvable cyclic metric Lie algebra $(\mathfrak{r},\mathbf{B}|_{\mathfrak{r}})$.
\end{thm}
\begin{thm}\label{1-thm-main2}
Let  $(\mathfrak{g},\mathbf{B})$ be a complex  cyclic metric Lie algebra with a Levi-decomposition $\mathfrak{g}=\mathfrak{s}+\mathfrak{r}$.

(a) Suppose that $\tilde{\mathfrak{s}}\subset \mathfrak{s}$ is a simple factor  not isomorphic to $\mathfrak{sl}(2,\mathbb{C})$, then $\mathbf{B}(\tilde{\mathfrak{s}},\mathfrak{g})=\mathbf{B}([\tilde{\mathfrak{s}},\mathfrak{g}],\mathfrak{g})=0$.

(b)  Suppose that $\tilde{\mathfrak{s}}=\mathfrak{sl}(2,\mathbb{C})\oplus\cdots\oplus\mathfrak{sl}(2,\mathbb{C})\subset \mathfrak{s}$ is a semisimple factor,
 and the dimension of every $\tilde{\mathfrak{s}}$-irreducible subspace of  $\mathfrak{r}$ is not equal to $3$,  then $\mathbf{B}(\tilde{\mathfrak{s}},\mathfrak{r})=\mathbf{B}([\tilde{\mathfrak{s}},\mathfrak{r}],\mathfrak{r})=0$.

 (c) Let $\mathfrak{r}=\mathfrak{a}+\mathrm{nil}\,\mathfrak{g}$ be the decomposition under the action of $\mathfrak{s}$, where $[\mathfrak{s},\mathfrak{a}]=0$ and $\mathrm{nil}\,\mathfrak{g}$ is the nilradical of $\mathfrak{g}$.
 Suppose that $[\mathfrak{a},\mathrm{nil}\,\mathfrak{g}]=\mathrm{nil}\,\mathfrak{g}$,
 then $\mathbf{B}(\mathfrak{s},\mathfrak{r})=\mathbf{B}([\mathfrak{s},\mathfrak{r}],\mathfrak{r})=0$.
\end{thm}

Our central result provides a complete structural characterization of  non-semisimple, non-solvable  non-degenerate cyclic metric Lie algebras. To state it, we introduce the following construction,
 a cyclic version of the double extension process \cite{mr85}.
Let $(\mathfrak{h},\mathbf{B}_{\mathfrak{h}})$ and $(\mathfrak{s},\mathbf{B}_{\mathfrak{s}})$ be two non-degenerate cyclic metric Lie algebras, and let $\mathfrak{h}+_{\theta}\mathfrak{s}$ be a Lie algebra obtained by the central extension of $\mathfrak{h}$ with respect to a cocycle $\theta\in(\wedge^{2}\mathfrak{h}^{\ast})\otimes\mathfrak{s}$.
Let
$\pi:\mathfrak{s}\rightarrow \mathrm{Der}(\mathfrak{h}+_{\theta}\mathfrak{s})$ be a Lie algebra homomorphism
such that
$\pi(x)(\mathfrak{h})\subset \mathfrak{h}$ and
$\pi(x)(y)=[x,y]_{\mathfrak{s}}$ for all $x,y\in\mathfrak{s}$. Then the cyclic metric Lie algebra $\left(\mathfrak{s}+_{\pi}(\mathfrak{h}+_{\theta}\mathfrak{s}),\mathbf{B}\right)$ defined by
\begin{eqnarray*}
\mathbf{B}((x_{1},h_{1},y_{1}),(x_{2},h_{2},y_{2}))=\widetilde{\mathbf{B}}(x_{1},x_{2})
+\mathbf{B}_{\mathfrak{h}}(h_{1},h_{2})+\mathbf{B}_{\mathfrak{s}}(x_{1},y_{2})
+\mathbf{B}_{\mathfrak{s}}(y_{1},x_{2}),
\end{eqnarray*}
 for all $x_{i},y_{i}\in \mathfrak{s}$,  $h_{i}\in\mathfrak{h}$, $i=1,2$, is called a double extension of $(\mathfrak{h},\mathbf{B}_{\mathfrak{h}})$ with respect to $\mathfrak{s}$, where $\widetilde{\mathbf{B}}$ is an arbitrary cyclic metric on $\mathfrak{s}$ (possibly degenerate), and $\theta$ should satisfy the identity
\begin{equation*}
  \mathbf{B}_{\mathfrak{s}}(x,\theta(h_{1},h_{2}))=\mathbf{B}_{\mathfrak{h}}([x,h_{2}],h_{1})
-\mathbf{B}_{\mathfrak{h}}([x,h_{1}],h_{2}),\quad \forall x\in \mathfrak{s},h_{1},h_{2}\in \mathfrak{h}.
\end{equation*}
In particular, when $\dim \mathfrak{s}=1$ and $\widetilde{\mathbf{B}}=0$, the cyclic metric Lie algebra
$\left(\mathfrak{s}+_{\pi}(\mathfrak{h}+_{\theta}\mathfrak{s}),\mathbf{B}\right)$
is called a one-dimensional central double extension of $(\mathfrak{h},\mathbf{B}_{\mathfrak{h}})$.
By the same proof as in \cite{fs87},
one can show that every $n$-dimensional  non-degenerate cyclic metric Lie algebra
with a non-zero degenerate center is a  one-dimensional central double extension of an
$(n-2)$-dimensional  non-degenerate cyclic metric Lie algebra.

\begin{thm}\label{1-thm-main3}
Let $(\mathfrak{g},\mathbf{B})$ be a non-degenerate cyclic metric Lie algebra  with   a Levi-decomposition $\mathfrak{g}=\mathfrak{s}+\mathfrak{r}$, over $\mathbb{C}$ or $\mathbb{R}$.

(a) If $[\mathfrak{s},C(\mathrm{nil}\,\mathfrak{g})]=0$, then $[\mathfrak{s},\mathfrak{r}]=0$ and $(\mathfrak{g},\mathbf{B})$ is the orthogonal direct product of $(\mathfrak{s},\mathbf{B}|_{\mathfrak{s}})$
 and  $(\mathfrak{r},\mathbf{B}|_{\mathfrak{r}})$.

(b) If $[\mathfrak{s},C(\mathrm{nil}\,\mathfrak{g})]\neq0$, then we have $\mathbf{B}([\mathfrak{s},C(\mathrm{nil}\,\mathfrak{g})],\mathfrak{r})=0$ and

$(b_1)$ over $\mathbb{C}$, $(\mathfrak{g},\mathbf{B})$  is a double extension of a  cyclic metric  Lie algebra with respect to $\mathfrak{sl}(2,\mathbb{C})$;

$(b_2)$ over $\mathbb{R}$, $(\mathfrak{g},\mathbf{B})$  is a double extension of a  cyclic metric  Lie algebra with respect to $\mathfrak{sl}(2,\mathbb{R})$, $\mathfrak{su}(2)$ or $\mathfrak{sl}(2,\mathbb{C})$ (regarded as a real Lie algebra).
\end{thm}

Through out the paper, we make the following assumptions.
$\mathbb{F}$ is a field with $\mathrm{char}\, \mathbb{F}\neq 2,3$.
$\mathfrak{g}$ is a Lie algebra over $\mathbb{F}$, unless otherwise mentioned.
$\mathfrak{g}=\mathfrak{s}+\mathfrak{r}$ is  a Levi-decomposition of $\mathfrak{g}$,
where $\mathfrak{s}$ denotes a semisimple Levi-factor and $\mathfrak{r}$  the radical of $\mathfrak{g}$. Every finite dimensional representation of $\mathfrak{s}$ is completely reducible.
$\mathfrak{r}=\mathfrak{a}+\mathrm{nil}\,\mathfrak{g}$ is the decomposition under the action of $\mathfrak{s}$,
where $[\mathfrak{s},\mathfrak{a}]=0$ and $\mathrm{nil}\,\mathfrak{g}$ is the nilradical of $\mathfrak{g}$.
As usual,  $\mathfrak{a}$ acts   diagonally on $\mathrm{nil}\,\mathfrak{g}$ over $\mathbb{C}$.
$C(\mathfrak{g})$ and $C(\mathrm{nil}\,\mathfrak{g})$ denote the centers of $\mathfrak{g}$ and $\mathrm{nil}\,\mathfrak{g}$, respectively.
$\langle\cdot,\cdot\rangle$ is a symmetric bilinear form (not necessarily cyclic),
and $\mathbf{B}$ is a cyclic metric.

This paper is organized as follows. In Section \ref{sec2}, we review basic definitions and prove a key lemma. Section \ref{sec3} is devoted to the study of decomposable cyclic metric Lie algebras. In Section \ref{sec4}, we introduce cyclic quadruples and use them to analyze the interaction between semisimple and radical parts. Finally, in Section \ref{sec5}, we prove Theorem \ref{1-thm-main3}.

\section{Basic definitions and a key lemma}\label{sec2}
A metric Lie algebra is a pair $(\mathfrak{g},\langle\cdot,\cdot\rangle)$, where  $(\mathfrak{g},[\cdot,\cdot])$ is a Lie algebra and  $\langle\cdot,\cdot\rangle$ is a symmetric  bilinear  form on $\mathfrak{g}$. For every  subspace $V$ of  $(\mathfrak{g}, \langle\cdot,\cdot\rangle)$, denote by $V^{\perp}=\{x\in \mathfrak{g}|\langle x,V\rangle=0\}$, the orthogonal complement of $V$ in $\mathfrak{g}$. The subspace $V$ is called isotropic if $V\subset V^{\perp}$.
The metric $\langle\cdot,\cdot\rangle$ or the metric Lie algebra $(\mathfrak{g},\langle\cdot,\cdot\rangle)$  is called non-degenerate
if $\mathfrak{g}^{\perp}=\{0\}$.
We say $(\mathfrak{g}, \langle\cdot,\cdot\rangle)$ is an orthogonal direct product of $(\mathfrak{g}_{1},\langle\cdot,\cdot\rangle|_{\mathfrak{g}_{1}})$ and
$(\mathfrak{g}_{2},\langle\cdot,\cdot\rangle|_{\mathfrak{g}_{2}})$, if $\mathfrak{g}=\mathfrak{g}_{1}\oplus\mathfrak{g}_{2}$ is a direct sum of two ideals $\mathfrak{g}_{1}$ and $\mathfrak{g}_{2}$,
and $\langle\mathfrak{g}_{1},\mathfrak{g}_{2}\rangle=0$.
\begin{defn}
The index of a metric Lie algebra  $(\mathfrak{g},\langle\cdot,\cdot\rangle)$ is defined as
$$\max\{\dim V|V\,\mathrm{is}\, \mathrm{isotropic}\},$$
and is  denoted by $\mathrm{Ind}(\mathfrak{g},\langle\cdot,\cdot\rangle)$.
In particular, it is called Riemannian, Lorentz and trans-Lorentz if the index is 0,  1 and 2, respectively.
\end{defn}

By definition, we can easily derive a decomposition criterion for cyclic metric Lie algebras.
\begin{thm}\label{2-thm-abc}
Let $\mathfrak{g}$ be a  Lie algebra and  $\mathfrak{g}=\mathfrak{h}+\mathfrak{i}$ be a direct sum of  vector spaces,  where $\mathfrak{h}$ is a subalgebra of $\mathfrak{g}$ and $\mathfrak{i}$ is an ideal of $\mathfrak{g}$. Then a metric $\langle\cdot,\cdot\rangle$ on $\mathfrak{g}$ is cyclic if and only if the following conditions hold.

(a) The metric Lie algebras $(\mathfrak{h},\langle\cdot,\cdot\rangle|_{\mathfrak{h}})$ and  $(\mathfrak{i},\langle\cdot,\cdot\rangle|_{\mathfrak{i}})$ are both cyclic.

(b) $\langle[i_{1},i_{2}],h\rangle+\langle[i_{2},h],i_{1}\rangle+\langle[h,i_{1}],i_{2}\rangle=0$
for all $i_{1},i_{2}\in \mathfrak{i}$, $h\in \mathfrak{h}$.

(c) $\langle[h_{1},h_{2}],i\rangle+\langle[h_{2},i],h_{1}\rangle+\langle[i,h_{1}],h_{2}\rangle=0$
for all $h_{1},h_{2}\in \mathfrak{h}$, $i\in \mathfrak{i}$.
\end{thm}

\begin{cor}\label{2-cor-decom-abelian}
With notation as in Theorem \ref{2-thm-abc}, assume that the ideal $\mathfrak{i}$ is Abelian. Then
a metric $\langle\cdot,\cdot\rangle$ on $\mathfrak{g}=\mathfrak{h}+\mathfrak{i}$ is cyclic if and only if the following conditions hold.

(a)  $(\mathfrak{h},\langle\cdot,\cdot\rangle|_{\mathfrak{h}})$ is cyclic.

(b) $\langle[i_{2},h],i_{1}\rangle+\langle[h,i_{1}],i_{2}\rangle=0$
for all $i_{1},i_{2}\in \mathfrak{i}$, $h\in \mathfrak{h}$.

(c) $\langle[h_{1},h_{2}],i\rangle+\langle[h_{2},i],h_{1}\rangle+\langle[i,h_{1}],h_{2}\rangle=0$
for all $h_{1},h_{2}\in \mathfrak{h}$, $i\in \mathfrak{i}$.
\end{cor}

We now state some useful lemmas.
\begin{lem}\label{2-lem-subalgebra}
Let $(\mathfrak{g},\mathbf{B})$ be a cyclic metric Lie algebra and $\mathfrak{h}$ be a Lie subalgebra of $\mathfrak{g}$.
 Then $(\mathfrak{h},\mathbf{B}|_{\mathfrak{h}})$ is cyclic.
\end{lem}

\begin{lem}\label{2-lem-decomposition}
Let $(\mathfrak{g},\mathbf{B})$ be a cyclic metric Lie algebra. Then for every  Lie subalgebra $\mathfrak{h}$ of $\mathfrak{g}$ and vector space $V\subset \mathfrak{g}$ satisfying $[\mathfrak{h},V]=0$, we have
$\mathbf{B}([\mathfrak{h},\mathfrak{h}],V)=0$.
In particular, if in addition that $\mathfrak{h}$ is semisimple, then  $\mathbf{B}(\mathfrak{h},V)=0$.
\end{lem}
\begin{proof}
For all $x,y\in\mathfrak{h}$ and $v\in V$, we have
\begin{eqnarray*}
0&=&\mathbf{B}([x,y],v)+\mathbf{B}([y,v],x)+\mathbf{B}([v,x],y)\\
 &=&\mathbf{B}([x,y],v),
\end{eqnarray*}
since $[y,v]=[v,x]=0$.
This proves the lemma.
\end{proof}

The following corollaries follow easily from Lemma \ref{2-lem-decomposition}.
\begin{cor}\label{2-cor-isotropic}
Let $(\mathfrak{g}, \mathbf{B})$ be a cyclic metric Lie algebra.
Then for every subalgebra $\mathfrak{h}$ of $\mathfrak{g}$,
$\mathbf{B}([\mathfrak{h},\mathfrak{h}],C(\mathfrak{h}))=0$, where $C(\mathfrak{h})$ denotes the center of $\mathfrak{h}$. Hence the space $C(\mathfrak{h})\cap [\mathfrak{h},\mathfrak{h}]$ is isotropic.
\end{cor}

Corollary \ref{2-cor-isotropic} asserts that every $n$-dimensional non-Abelian nilpotent Lie algebra endowed with a non-degenerate cyclic metric must be a one-dimensional central extension of an $(n-2)$-dimensional nilpotent
cyclic metric Lie algebra.
\begin{cor}\label{2-cor-twodirectsum}
Let $(\mathfrak{g}, \mathbf{B})$ be a cyclic metric Lie algebra.
Assume that $\mathfrak{g}=\mathfrak{g}_{1}\oplus\mathfrak{g}_{2}$ is a direct sum of Lie algebras $\mathfrak{g}_{1}$ and $\mathfrak{g}_{2}$, where $\mathfrak{g}_{1}$ is semisimple, then
$\mathbf{B}(\mathfrak{g}_{1},\mathfrak{g}_{2})=0$ and consequently $(\mathfrak{g}, \mathbf{B})$ is an orthogonal direct product of $(\mathfrak{g}_{1},\mathbf{B}|_{\mathfrak{g}_{1}})$ and
$(\mathfrak{g}_{2},\mathbf{B}|_{\mathfrak{g}_{2}})$.
\end{cor}
\begin{cor}\label{2-cor-reductive}
Every reductive cyclic metric Lie algebra is an orthogonal direct product of simple cyclic metric Lie algebras and the Abelian part.
\end{cor}

\begin{prop}
Let $(\mathfrak{g}, \mathbf{B})$ be a cyclic metric Lie algebra.
If  $\mathfrak{h}$ is an ideal of $\mathfrak{g}$,
then  $\mathfrak{h}^{\perp}$ is a Lie subalgebra of $\mathfrak{g}$ and
$\dim (\mathfrak{h}\cap \mathfrak{h}^{\perp})\leq \mathrm{Ind}(\mathfrak{g},\mathbf{B})$.
In particular, $\dim \mathfrak{g}^{\perp}\leq \mathrm{Ind}(\mathfrak{g},\mathbf{B})$.
\end{prop}
\begin{proof}
 For all $x,y\in\mathfrak{h}^{\perp}$ and  $z\in\mathfrak{h}$, we have $[x,z],[y,z]\in \mathfrak{h}$.
  Hence
\begin{equation*}
\mathbf{B}([x,y],z)=-\mathbf{B}([y,z],x)-\mathbf{B}([z,x],y)=0,
\end{equation*}
which implies  $[x,y]\in \mathfrak{h}^{\perp}$,  so $\mathfrak{h}^{\perp}$ is a Lie  subalgebra of $\mathfrak{g}$. The inequality follows from the definition of the index.
\end{proof}

From condition (a) of Theorem \ref{2-thm-abc}, we make the following definition.
\begin{defn}
A cyclic metric Lie algebra $(\mathfrak{g},\mathbf{B})$ is called indecomposable if it does not admit any non-degenerate proper ideal. Otherwise, it is called decomposable.
\end{defn}
By Corollary \ref{2-cor-isotropic}, we obtain
\begin{prop}
Let $(\mathfrak{g},\mathbf{B})$ be an indecomposable cyclic metric Lie algebra. Then $C(\mathfrak{g})$ is isotropic and
$\dim C(\mathfrak{g})\leq \mathrm{Ind}(\mathfrak{g},\mathbf{B})$.
\end{prop}

We now state a key lemma, which will be used frequently in this paper.
\begin{lem}\label{2-prop-semisimple}
Let $\mathfrak{g}$ be a  Lie algebra
and $(V,\mathbf{K})$ be a vector space  endowed with a symmetric bilinear
 form $\mathbf{K}:V\times V\rightarrow \mathbb{F}$.
Assume that $\pi:\mathfrak{g}\rightarrow \mathfrak{gl}(V,\mathbb{F})$ is a representation of $\mathfrak{g}$ such that for each $x\in\mathfrak{g}$, $\pi(x):V\rightarrow V$ is symmetric with respect to $\mathbf{K}$,
 namely, $\mathbf{K}(\pi(x)(u),v)=\mathbf{K}(u,\pi(x)(v))$ holds for all $u,v\in V$.
 Then $\mathbf{K}(\pi([x,y])(u),v)=0$ for all $x,y\in \mathfrak{g}$ and $u,v\in V$.
 In particular, if $\mathfrak{g}$ is semisimple, then $\mathbf{K}=0$ when $\pi$ has no trivial subrepresentation,
 and  $\pi=0$ when $\mathbf{K}$ is non-degenerate.
\end{lem}
\begin{proof}
For all $x,y\in\mathfrak{g}$, $u,v\in V$, we have
\begin{eqnarray*}
\mathbf{K}(\pi([x,y])(u),v)&=&\mathbf{K}([\pi(x)\pi(y)-\pi(y)\pi(x)](u),v)\\
          &=&\mathbf{K}(\pi(x)\pi(y)(u),v)-\mathbf{K}(\pi(y)\pi(x)(u),v)\\
          &=& \mathbf{K}(u,\pi(y)\pi(x)(v))-\mathbf{K}(u,\pi(x)\pi(y)(v))\\
          &=&\mathbf{K}(u,[\pi(y),\pi(x)](v))\\
          &=&-\mathbf{K}(\pi([x,y])(u),v),
\end{eqnarray*}
so $\mathbf{K}(\pi([x,y])(u),v)=0$. The remaining claims follow immediately.
\end{proof}
\begin{cor}\label{2-cor-abelian}
Let $(\mathfrak{g},\mathbf{B})$ be a cyclic metric Lie algebra.
If $\mathfrak{h}$ is an Abelian ideal of $\mathfrak{g}$, then
$\mathbf{B}([[\mathfrak{g},\mathfrak{g}],\mathfrak{h}],\mathfrak{h})=0.$
\end{cor}
\begin{proof}
Define the representation  $\pi:\mathfrak{g}\rightarrow \mathfrak{gl}(\mathfrak{h},\mathbb{F})$
 by $\pi(x)(y)=[x,y]$ for all $x\in \mathfrak{g}, y\in \mathfrak{h}$.
Since $\mathfrak{h}$ is Abelian, then for each $x\in \mathfrak{g}$, $\pi(x)$ is symmetric with respect to
$\mathbf{B}|_{\mathfrak{h}}$. The result follows from Lemma \ref{2-prop-semisimple}.
\end{proof}

\section{Decomposable cyclic metric Lie algebras}\label{sec3}
In this section, we study decomposable cyclic metric Lie algebras.
Let $(\mathfrak{g}_{1},[\cdot,\cdot]_{1},\mathbf{B}_{1})$
 and $(\mathfrak{g}_{2},[\cdot,\cdot]_{2},\mathbf{B}_{2})$ be two cyclic metric Lie algebras.
 Suppose that $\pi: \mathfrak{g}_{1}\rightarrow \mathrm{Der}(\mathfrak{g}_{2})$ is a Lie algebra homomorphism,
 and define the Lie algebra $\mathfrak{g}_{1}+_{\pi}\mathfrak{g}_{2}$ by
\begin{equation*}
[(x_{1},x_{2}),(y_{1},y_{2})]=([x_{1},y_{1}]_{1},[x_{2},y_{2}]_{2}+\pi(x_{1})(y_{2})-\pi(y_{1})(x_{2})),
\end{equation*}
for all $x_{1},y_{1}\in\mathfrak{g}_{1}$, $x_{2},y_{2}\in\mathfrak{g}_{2}$. Moreover, define the metric $\langle\cdot,\cdot\rangle=\mathbf{B}_{1}+\mathbf{B}_{2}$ on $\mathfrak{g}_{1}+_{\pi}\mathfrak{g}_{2}$ by
\begin{equation*}
\langle (x_{1},x_{2}),(y_{1},y_{2})\rangle=\mathbf{B}_{1}(x_{1},y_{1})+\mathbf{B}_{2}(x_{2},y_{2}), \forall x_{1},y_{1}\in\mathfrak{g}_{1}, x_{2},y_{2}\in\mathfrak{g}_{2}.
\end{equation*}
The resulting metric Lie algebra $(\mathfrak{g},\langle\cdot,\cdot\rangle)
=(\mathfrak{g}_{1},\mathbf{B}_{1})
+_{\pi}(\mathfrak{g}_{2},\mathbf{B}_{2})$ is called an orthogonal
semidirect product of $(\mathfrak{g}_{1},\mathbf{B}_{1})$ and
$(\mathfrak{g}_{2},\mathbf{B}_{2})$.
In particular, if  $\pi=0$,
then $(\mathfrak{g},\langle\cdot,\cdot\rangle)
=(\mathfrak{g}_{1},\mathbf{B}_{1})
\oplus(\mathfrak{g}_{2},\mathbf{B}_{2})$ is the orthogonal
direct product of $(\mathfrak{g}_{1},\mathbf{B}_{1})$  and $(\mathfrak{g}_{2},\mathbf{B}_{2})$.
\begin{prop}[\cite{ggo15}]\label{2-prop-pi}
The metric Lie algebra $(\mathfrak{g}_{1}+_{\pi}\mathfrak{g}_{2}, \langle\cdot,\cdot\rangle)$ is cyclic if and only if for each $x\in\mathfrak{g}_{1}$, the derivation $\pi(x)$ on $\mathfrak{g}_{2}$ is symmetric with respect to $\mathbf{B}_{2}$.
\end{prop}
Conversely, we have the following description of decomposable cyclic metric Lie algebras.
\begin{thm}\label{2-thm-decom}
Let $(\mathfrak{g},\mathbf{B})$ be a cyclic metric Lie algebra. Assume that $\mathfrak{g}_{2}$ is an ideal of $\mathfrak{g}$ and
 the restriction of $\mathbf{B}$ to $\mathfrak{g}_{2}$ is non-degenerate.
 Then $(\mathfrak{g},\mathbf{B})
=(\mathfrak{g}_{1},\mathbf{B}|_{\mathfrak{g}_{1}})
+_{\pi}(\mathfrak{g}_{2},\mathbf{B}|_{\mathfrak{g}_{2}})$, where
 $\mathfrak{g}_{1}=\mathfrak{g}_{2}^{\perp}$ is the $\mathbf{B}$-orthogonal complement of $\mathfrak{g}_{2}$ in  $\mathfrak{g}$,
and $\pi:\mathfrak{g}_{1}\rightarrow \mathrm{Der}(\mathfrak{g}_{2})$ is given by $\pi(x)(y)=[x,y]$
for all $ x\in\mathfrak{g}_{1}$, $y\in\mathfrak{g}_{2}$. In particular, $\mathfrak{g}_{1}$ is isomorphic to the quotient Lie algebra $\mathfrak{g}/\mathfrak{g}_{2}$, and $[\mathfrak{g}_{1},\mathfrak{g}_{1}]\subset \ker \pi$.
\end{thm}
\begin{proof}
Since $\mathbf{B}|_{\mathfrak{g}_{2}}$ is non-degenerate, we have  $\mathfrak{g}=\mathfrak{g}_{1}+\mathfrak{g}_{2}$, as a direct sum of vector spaces. For any $x=x_{1}+x_{2}$, $y=y_{1}+y_{2}\in\mathfrak{g}$ with $x_{1},y_{1}\in\mathfrak{g}_{1}$ and  $x_{2},y_{2}\in\mathfrak{g}_{2}$,
 we have
\begin{eqnarray*}
[x,y]&=&[x_{1}+x_{2},y_{1}+y_{2}]\\
     &=&[x_{1},y_{1}]+[x_{2},y_{2}]+[x_{1},y_{2}]-[y_{1},x_{2}],
\end{eqnarray*}
which implies that $\mathfrak{g}=\mathfrak{g}_{1}+_{\pi}\mathfrak{g}_{2}$.
Hence   $(\mathfrak{g},\mathbf{B})
=(\mathfrak{g}_{1},\mathbf{B}|_{\mathfrak{g}_{1}})
+_{\pi}(\mathfrak{g}_{2},\mathbf{B}|_{\mathfrak{g}_{2}})$.
The quotient map $\rho: \mathfrak{g}_{1}\rightarrow \mathfrak{g}/\mathfrak{g}_{2}$ defined by $\rho(x)=x+\mathfrak{g}_{2}$ for all $x\in \mathfrak{g}_{1}$, is an isomorphism.
The last claim follows by Lemma \ref{2-prop-semisimple} and Proposition \ref{2-prop-pi}.
\end{proof}

\begin{rem}
The above results show that Corollary \ref{2-cor-twodirectsum} does not hold for solvable cyclic metric Lie algebras.
For example, let $\mathfrak{g}=\mathfrak{g}_{1}\oplus\mathfrak{g}_{2}$ be the direct sum of two isomorphic two-dimensional real solvable Lie algebras   $\mathfrak{g}_{1}=\mathrm{span}\{x_{1},y_{1}\}$ and $\mathfrak{g}_{2}=\mathrm{span}\{x_{2},y_{2}\}$ with non-zero brackets
$$[x_{1},y_{1}]=y_{1},\quad [x_{2},y_{2}]=y_{2}.$$
Define a Lorentz cyclic metric $\mathbf{B}$  on $\mathfrak{g}$  by
\begin{eqnarray*}
\left\{
\begin{aligned}
  &\mathbf{B}(x_{1},x_{1})=\mathbf{B}(x_{2},x_{2})=0,\quad \mathbf{B}(x_{1},x_{2})=1, \\
  &\mathbf{B}(y_{1},y_{1})=\mathbf{B}(y_{2},y_{2})=1,\quad \mathbf{B}(y_{1},y_{2})=0,\\
  &\mathbf{B}(x_{i},y_{j})=0,\quad 1\leq i,j\leq 2.
\end{aligned}
\right.
\end{eqnarray*}
It is easy to  verify  that
$(\mathfrak{g},\mathbf{B})=(\mathfrak{h}_{1},\mathbf{B}|_{\mathfrak{h}_{1}})
+_{\pi}(\mathfrak{h}_{2},\mathbf{B}|_{\mathfrak{h}_{2}})$,
where  $\mathfrak{h}_{1}=\mathrm{span}\{x_{1},x_{2}\}$ and $\mathfrak{h}_{2}=\mathrm{span}\{y_{1},y_{2}\}$.
 However,  the cyclic metric Lie algebra $(\mathfrak{g},\mathbf{B})$ is not the orthogonal direct product of $(\mathfrak{g}_{1},\mathbf{B}|_{\mathfrak{g}_{1}})$ and
$(\mathfrak{g}_{2},\mathbf{B}|_{\mathfrak{g}_{2}})$.
\end{rem}

\begin{thm}\label{2.3-thm-sol-nondege}
Let $\mathfrak{g}$ be a  Lie algebra with a Levi-decomposition $\mathfrak{g}=\mathfrak{s}+\mathfrak{r}$. Assume that $\mathfrak{g}$ admits a  cyclic metric
$\mathbf{B}$  such that the restriction of $\mathbf{B}$ to $\mathfrak{r}$ is non-degenerate.  Then $[\mathfrak{s},\mathfrak{r}]=0$ and
 $(\mathfrak{g},\mathbf{B})
 =(\mathfrak{s},\mathbf{B}|_{\mathfrak{s}})
 \oplus(\mathfrak{r},\mathbf{B}|_{\mathfrak{r}})$
  is the orthogonal direct product of semisimple cyclic metric Lie algebra
$(\mathfrak{s},\mathbf{B}|_{\mathfrak{s}})$
 and solvable cyclic metric Lie algebra $(\mathfrak{r},\mathbf{B}|_{\mathfrak{r}})$.
\end{thm}
\begin{proof}
By Theorem \ref{2-thm-decom},
$(\mathfrak{g},\mathbf{B})
=(\mathfrak{s},\mathbf{B}|_{\mathfrak{s}})+_{\pi}(\mathfrak{r},\mathbf{B}|_{\mathfrak{r}})$,
where $\mathfrak{s}=\mathfrak{r}^{\perp}$ is isomorphic to the quotient Lie algebra  $\mathfrak{g}/\mathfrak{r}$
and hence semisimple. By Lemma \ref{2-prop-semisimple}, $\pi=0$, so the result follows.
\end{proof}

\begin{cor}\label{2.3-cor-non-degenerate}
Let $(\mathfrak{g},\mathbf{B})$ be a cyclic metric Lie algebra.
If the restriction of $\mathbf{B}$ to $\mathrm{nil}\,\mathfrak{g}$ is non-degenerate,
 then
 $(\mathfrak{g},\mathbf{B})
  =(\mathfrak{s},\mathbf{B}|_{\mathfrak{s}})
 \oplus(\mathfrak{r},\mathbf{B}|_{\mathfrak{r}})$.
\end{cor}
\begin{proof}
By Theorem \ref{2-thm-decom},
$(\mathfrak{g},\mathbf{B})
=((\mathrm{nil}\,\mathfrak{g})^{\perp},\mathbf{B}|_{(\mathrm{nil}\,\mathfrak{g})^{\perp}})
+_{\pi}(\mathrm{nil}\,\mathfrak{g},\mathbf{B}|_{\mathrm{nil}\,\mathfrak{g}})$,
where $(\mathrm{nil}\,\mathfrak{g})^{\perp}\cong\mathfrak{g}/\mathrm{nil}\,\mathfrak{g}$ is reductive.
By Corollary \ref{2-cor-reductive}, $((\mathrm{nil}\,\mathfrak{g})^{\perp},\mathbf{B}|_{(\mathrm{nil}\,\mathfrak{g})^{\perp}})
=(\mathfrak{s},\mathbf{B}|_{\mathfrak{s}})\oplus(\mathfrak{a},\mathbf{B}|_{\mathfrak{a}})$
is an orthogonal direct product of  a semisimple cyclic metric Lie algebra $(\mathfrak{s},\mathbf{B}|_{\mathfrak{s}})$ and  an Abelian metric Lie algebra $(\mathfrak{a},\mathbf{B}|_{\mathfrak{a}})$.
By Lemma \ref{2-prop-semisimple}, it is easily seen that $[\mathfrak{s},\mathrm{nil}\,\mathfrak{g}]=0$ and consequently  $[\mathfrak{s},\mathfrak{a}+\mathrm{nil}\,\mathfrak{g}]=0$.
Since
$\mathfrak{r}=\mathfrak{a}+\mathrm{nil}\,\mathfrak{g}$,
the result follows by Corollary \ref{2-cor-twodirectsum}.
\end{proof}

We now demonstrate the orthogonal property of $\mathfrak{s}$ and $\mathfrak{r}$ for Lorentz and trans-Lorentz cyclic metric  Lie algebras.
The following proposition shows that $[\mathfrak{s},\mathfrak{r}]\neq 0$ for trans-Lorentz cyclic metric Lie algebras
\begin{prop}\label{2-lem-sl(2)}
Let $\mathfrak{g}=\mathfrak{sl}(2,\mathbb{F})+_{\pi}\mathbb{F}^{2}$ be the  Lie algebra defined by the natural $\mathfrak{sl}(2,\mathbb{F})$-action on $\mathbb{F}^{2}$. Suppose that $\mathbf{B}$ is a cyclic metric on  $\mathfrak{g}$, then $\mathbf{B}(\mathbb{F}^{2},\mathfrak{g})=0$.
Consequently, every cyclic metric on $\mathfrak{g}$ is obtained by a trivial extension of a cyclic metric on
$\mathfrak{sl}(2,\mathbb{F})$.
\end{prop}
\begin{proof}
For all $x\in \mathfrak{sl}(2,\mathbb{F})\subset \mathfrak{g}$ and $y,z\in \mathbb{F}^{2}\subset \mathfrak{g}$, we have
\begin{eqnarray*}
0&=&\mathbf{B}([y,z],x)+\mathbf{B}([z,x],y)+\mathbf{B}([x,y],z)\\
 &=&-\mathbf{B}([x,z],y)+\mathbf{B}([x,y],z),
\end{eqnarray*}
which implies that $\mathrm{ad}\, x|_{\mathbb{F}^{2}}$ is symmetric. Then by Lemma \ref{2-prop-semisimple}, we get
$$0=\mathbf{B}([\mathfrak{sl}(2,\mathbb{F}),\mathbb{F}^{2}],\mathbb{F}^{2})
=\mathbf{B}(\mathbb{F}^{2},\mathbb{F}^{2}).$$

Now let $\{H,X,Y\}$ be the basis of $\mathfrak{sl}(2,\mathbb{F})$ given by
\begin{eqnarray*}
H=\left(
                  \begin{array}{cc}
                    1 & 0 \\
                    0 & -1 \\
                  \end{array}
                \right),\quad
 X=\left(
            \begin{array}{cc}
              0 & 1 \\
              0 & 0 \\
            \end{array}
          \right),\quad
 Y=\left(
            \begin{array}{cc}
              0 & 0 \\
              1 & 0 \\
            \end{array}
          \right).
\end{eqnarray*}
It is straightforward to verify that
\begin{eqnarray*}
[H,X]=2X,\quad [H,Y]=-2Y,\quad [X,Y]=H.
\end{eqnarray*}
Thus $\mathbf{B}|_{\mathfrak{sl}(2,\mathbb{F})}$ is cyclic if and only if $4\mathbf{B}(X,Y)+\mathbf{B}(H,H)=0$.

Let $e_{1}=(1,0)^{T},e_{2}=(0,1)^{T}\in \mathbb{F}^{2}$, one has
\begin{eqnarray*}
\left\{
\begin{aligned}
&[H,e_{1}]=e_{1},\quad  [H,e_{2}]=-e_{2},\\
&[X,e_{1}]=0,\quad [X,e_{2}]=e_{1},\\
&[Y,e_{1}]=e_{2},\quad [Y,e_{2}]=0.
\end{aligned}
\right.
\end{eqnarray*}
Note that
\begin{eqnarray*}
0&=&\mathbf{B}([H,X],e_{1})+\mathbf{B}([X,e_{1}],H)+\mathbf{B}([e_{1},H],X)\\
 &=&2\mathbf{B}(X,e_{1})-\mathbf{B}(e_{1},X)\\
 &=& \mathbf{B}(X,e_{1}),
\end{eqnarray*}
and
\begin{eqnarray*}
0&=& \mathbf{B}([H,Y],e_{2})+\mathbf{B}([Y,e_{2}],H)+\mathbf{B}([e_{2},H],Y)\\
 &=& -2\mathbf{B}(Y,e_{2})+\mathbf{B}(e_{2},Y)\\
 &=& -\mathbf{B}(Y,e_{2}),
\end{eqnarray*}
which implies  $\mathbf{B}(X,e_{1})=\mathbf{B}(Y,e_{2})=0.$
Moreover,
\begin{eqnarray*}
0&=&\mathbf{B}([H,X],e_{2})+\mathbf{B}([X,e_{2}],H)+\mathbf{B}([e_{2},H],X)\\
 &=& 2\mathbf{B}(X,e_{2})+\mathbf{B}(e_{1},H)+\mathbf{B}(e_{2},X)\\
 &=& 3\mathbf{B}(X,e_{2})+\mathbf{B}(H,e_{1}),
\end{eqnarray*}
and
\begin{eqnarray*}
0&=& \mathbf{B}([X,Y],e_{1})+\mathbf{B}([Y,e_{1}],X)+\mathbf{B}([e_{1},X],Y)\\
 &=& \mathbf{B}(H,e_{1})+\mathbf{B}(X,e_{2}),
\end{eqnarray*}
we obtain $\mathbf{B}(X,e_{2})=\mathbf{B}(H,e_{1})=0$.

Finally,
\begin{eqnarray*}
0&=& \mathbf{B}([H,Y],e_{1})+\mathbf{B}([Y,e_{1}],H)+\mathbf{B}([e_{1},H],Y)\\
 &=& -2\mathbf{B}(Y,e_{1})+\mathbf{B}(e_{2},H)-\mathbf{B}(e_{1},Y)\\
 &=& \mathbf{B}(H,e_{2})-3\mathbf{B}(Y,e_{1}),
\end{eqnarray*}
and
\begin{eqnarray*}
0&=& \mathbf{B}([X,Y],e_{2})+\mathbf{B}([Y,e_{2}],X)+\mathbf{B}([e_{2},X],Y)\\
 &=& \mathbf{B}(H,e_{2})-\mathbf{B}(e_{1},Y),
\end{eqnarray*}
which gives $\mathbf{B}(H,e_{2})=\mathbf{B}(Y,e_{1})=0$. This shows that $\mathbf{B}(\mathfrak{sl}(2,\mathbb{F}),\mathbb{F}^{2})=0$, completing the proof.
\end{proof}
\begin{cor}
Let $\mathfrak{g}=\mathfrak{gl}(2,\mathbb{F})+_{\pi}\mathbb{F}^{2}$ be the  Lie algebra defined by the natural $\mathfrak{gl}(2,\mathbb{F})$-action on $\mathbb{F}^{2}$. Suppose that $\mathbf{B}$ is a cyclic metric on  $\mathfrak{g}$, then $\mathbf{B}(\mathbb{F}^{2},\mathfrak{g})=0$.
Consequently, every cyclic metric on $\mathfrak{g}$ is obtained by a trivial extension of a cyclic metric on
$\mathfrak{gl}(2,\mathbb{F})$.
\end{cor}
\begin{proof}
By Proposition \ref{2-lem-sl(2)},  $\mathbf{B}(\mathbb{F}^{2},\mathfrak{sl}(2,\mathbb{F})+_{\pi}\mathbb{F}^{2})=0$.
Then for all $x\in \mathfrak{gl}(2,\mathbb{F})$, we have
\begin{eqnarray*}
  \mathbf{B}(\mathbb{F}^{2},x) &=& \mathbf{B}([\mathfrak{sl}(2,\mathbb{F}),\mathbb{F}^{2}],x) \\
   &=& -\mathbf{B}([\mathbb{F}^{2},x],\mathfrak{sl}(2,\mathbb{F}))
   -\mathbf{B}([x,\mathfrak{sl}(2,\mathbb{F})],\mathbb{F}^{2})\\
   &=& 0,
\end{eqnarray*}
which asserts that $\mathbf{B}(\mathbb{F}^{2},\mathfrak{g})=0$.
\end{proof}
\begin{prop}\label{2-lem-tran-lorentz}
Let $(\mathfrak{g},\mathbf{B})$ be a cyclic metric Lie algebra
 with a Levi-decomposition $\mathfrak{g}=\mathfrak{s}+\mathfrak{r}$.
If the restriction of $\mathbf{B}$ to $\mathrm{nil}\,\mathfrak{g}$ is  Lorentz or trans-Lorentz,
 then $\dim\, [\mathfrak{s},\mathfrak{r}]\leq2$ and $\mathbf{B}(\mathfrak{s},\mathfrak{r})=0$.
\end{prop}
\begin{proof}
Let $0\subset C^{1}(\mathrm{nil}\,\mathfrak{g})=C(\mathrm{nil}\,\mathfrak{g})\subset C^{2}(\mathrm{nil}\,\mathfrak{g})\subset\cdots\subset\mathrm{nil}\,\mathfrak{g}$ be the upper central series of $\mathrm{nil}\,\mathfrak{g}$, where  $C^{i+1}(\mathrm{nil}\,\mathfrak{g})/C^{i}(\mathrm{nil}\,\mathfrak{g})$ is the center of $\mathrm{nil}\,\mathfrak{g}/C^{i}(\mathrm{nil}\,\mathfrak{g})$. In particular, $[\mathfrak{s},C^{i}(\mathrm{nil}\,\mathfrak{g})]\subset C^{i}(\mathrm{nil}\,\mathfrak{g})$.

For all $x\in\mathfrak{s}$, $y,z\in C(\mathrm{nil}\,\mathfrak{g})$,
\begin{eqnarray*}
0&=& \mathbf{B}([y,z],x)+\mathbf{B}([z,x],y)+\mathbf{B}([x,y],z)\\
 &=& -\mathbf{B}([x,z],y)+\mathbf{B}([x,y],z),
\end{eqnarray*}
which shows that $\mathrm{ad}\, x|_{C(\mathrm{nil}\,\mathfrak{g})}$ is symmetric.
Thus by Lemma \ref{2-prop-semisimple}, we have
\begin{eqnarray*}
0&=&\mathbf{B}\left([[\mathfrak{s},\mathfrak{s}],C(\mathrm{nil}\,\mathfrak{g})],C(\mathrm{nil}\,\mathfrak{g})\right)\\
 &=&\mathbf{B}([\mathfrak{s},C(\mathrm{nil}\,\mathfrak{g})],C(\mathrm{nil}\,\mathfrak{g})).
\end{eqnarray*}
This implies that the space $[\mathfrak{s},C(\mathrm{nil}\,\mathfrak{g})]$ is isotropic  and hence $\mathrm{dim}\,[\mathfrak{s},C(\mathrm{nil}\,\mathfrak{g})]\leq 2$.
Since  $\mathfrak{s}$ is semisimple and every representation of $\mathfrak{s}$ is completely reducible,  the only non-trivial subrepresentation of $\mathfrak{s}$ on $C(\mathrm{nil}\,\mathfrak{g})$ is the natural action of $\mathfrak{sl}(2,\mathbb{F})$ on $\mathbb{F}^{2}$.
It follows from Lemma  \ref{2-lem-decomposition} and  Proposition \ref{2-lem-sl(2)}  that
$\mathbf{B}(\mathfrak{s},C(\mathrm{nil}\,\mathfrak{g}))=0$.

Furthermore, for all $x\in\mathfrak{s}$, $y,z\in C^{2}(\mathrm{nil}\,\mathfrak{g})$,
\begin{eqnarray*}
0&=& \mathbf{B}([y,z],x)+\mathbf{B}([z,x],y)+\mathbf{B}([x,y],z)\\
 &=& -\mathbf{B}([x,z],y)+\mathbf{B}([x,y],z),
\end{eqnarray*}
which shows that $\mathrm{ad}\, x|_{C^{2}(\mathrm{nil}\,\mathfrak{g})}$ is symmetric.
 By the same argument, we obtain  $\mathrm{dim}\,[\mathfrak{s},C^{2}(\mathrm{nil}\,\mathfrak{g})]\leq 2$ and
$\mathbf{B}(\mathfrak{s},C^{2}(\mathrm{nil}\,\mathfrak{g}))=0$.
Continuing this process, we eventually conclude that $\dim\,[\mathfrak{s},\mathrm{nil}\,\mathfrak{g}]\leq 2$ and $\mathbf{B}(\mathfrak{s},\mathrm{nil}\,\mathfrak{g})=0$.
Since  $[\mathfrak{s},\mathfrak{a}]=0$, it follows that
$\dim\,[\mathfrak{s},\mathfrak{r}]\leq 2$  and $\mathbf{B}(\mathfrak{s},\mathfrak{r})=0$.  This completes the proof.
\end{proof}

Theorem \ref{1-thm-main1.5} summarizes Theorem \ref{2.3-thm-sol-nondege},  Corollary \ref{2.3-cor-non-degenerate}, and the following theorem.
\begin{thm}\label{2.3-thm-lorentz}
Let $\mathfrak{g}$ be a  Lie algebra   with a Levi-decomposition $\mathfrak{g}=\mathfrak{s}+\mathfrak{r}$.
 Suppose that
 $\mathbf{B}$ is a cyclic metric on  $\mathfrak{g}$ such that
  $\mathbf{B}|_{\mathrm{nil}\,\mathfrak{g}}$ is Lorentz,
then $(\mathfrak{g},\mathbf{B})
 =(\mathfrak{s},\mathbf{B}|_{\mathfrak{s}})\oplus(\mathfrak{r},\mathbf{B}|_{\mathfrak{r}})$.
\end{thm}
\begin{proof}
 From the proof of   Proposition \ref{2-lem-tran-lorentz}, it follows that $\dim\,[\mathfrak{s},\mathfrak{r}]\leq 1$ and
$\mathbf{B}(\mathfrak{s},\mathfrak{r})=0$.
Hence $[\mathfrak{s},\mathfrak{r}]=0$ and $(\mathfrak{g},\mathbf{B})
 =(\mathfrak{s},\mathbf{B}|_{\mathfrak{s}})\oplus(\mathfrak{r},\mathbf{B}|_{\mathfrak{r}})$
 is the orthogonal direct product.
\end{proof}

Finally, we provide an example to show that Proposition \ref{2-lem-tran-lorentz}  does not hold for cyclic metric    Lie algebras  with index $\geq3$.
\begin{example}\label{3-exam-so(3)}
Let $\mathfrak{g}=\mathfrak{so}(3,\mathbb{F})$ be the $3$-dimensional  Lie algebra consisting of skew symmetric matrices with a basis $\{\mathbf{i},\mathbf{j},\mathbf{k}\}$ given by
$$\mathbf{i}=\left(
               \begin{array}{ccc}
                 0 & 0 & 0 \\
                 0 & 0 & -1 \\
                 0 & 1 & 0 \\
               \end{array}
             \right),\quad
\mathbf{j}=\left(
               \begin{array}{ccc}
                 0 & 0 & 1 \\
                 0 & 0 & 0 \\
                 -1 & 0 & 0 \\
               \end{array}
             \right),\quad
\mathbf{k}=\left(
               \begin{array}{ccc}
                 0 & -1 & 0 \\
                 1 & 0 & 0 \\
                 0 & 0 & 0 \\
               \end{array}
             \right).$$
Clearly, it is easy to see that
$$[\mathbf{i},\mathbf{j}]=\mathbf{k},\quad [\mathbf{j},\mathbf{k}]=\mathbf{i},\quad [\mathbf{k},\mathbf{i}]=\mathbf{j}.$$
Let $\mathfrak{g}=\mathfrak{so}(3,\mathbb{F})+_{\pi}\mathbb{F}^{3}$ be the  Lie algebra defined by the natural $\mathfrak{so}(3,\mathbb{F})$-action on $\mathbb{F}^{3}$.
Denote $e_{1}=(1,0,0)^{T},e_{2}=(0,1,0)^{T}, e_{3}=(0,0,1)^{T}\in \mathbb{F}^{3}$, one has
\begin{eqnarray*}
\begin{cases}
[\mathbf{i},e_{1}]=0,\quad  [\mathbf{i},e_{2}]=e_{3},\quad [\mathbf{i},e_{3}]=-e_{2},\\
[\mathbf{j},e_{1}]=-e_{3},\quad  [\mathbf{j},e_{2}]=0,\quad [\mathbf{j},e_{3}]=e_{1},\\
[\mathbf{k},e_{1}]=e_{2},\quad  [\mathbf{k},e_{2}]=-e_{1},\quad [\mathbf{k},e_{3}]=0.
\end{cases}
\end{eqnarray*}
In fact, the natural action of $\mathfrak{so}(3,\mathbb{F})$ on $\mathbb{F}^{3}$ is equivalent to the adjoint representation of $\mathfrak{so}(3,\mathbb{F})$ via the linear map $\varphi: \mathfrak{so}(3,\mathbb{F})\rightarrow \mathbb{F}^{3}$ defined by $\varphi(\mathbf{i})=e_{1}$, $\varphi(\mathbf{j})=e_{2}$ and $\varphi(\mathbf{k})=e_{3}$.

Now let $P$ and $Q$ be two   matrices of order 3, where $P$ is symmetric, and
$\langle\cdot,\cdot\rangle_{(P,Q)}$ be a   metric on $\mathfrak{g}=\mathfrak{so}(3,\mathbb{F})+_{\pi}\mathbb{F}^{3}$ defined by the following metric matrix
\begin{eqnarray*}
M=
\left(\begin{array}{cc}
P & Q \\
Q^{T} & 0
\end{array}\right),
\end{eqnarray*}
with respect to the basis $\{\mathbf{i},\mathbf{j},\mathbf{k},e_{1},e_{2},e_{3}\}$.
A straightforward computation shows that $\langle\cdot,\cdot\rangle_{(P,Q)}$ is cyclic if and only if  the following equations are satisfied
$$\left\{
  \begin{array}{ll}
  \mathrm{tr}(P)=0,&\\
\langle[\mathbf{j},\mathbf{k}],e_{i}\rangle_{(P,Q)}+\langle[\mathbf{k},e_{i}],\mathbf{j}\rangle_{(P,Q)}+ \langle[e_{i},\mathbf{j}],\mathbf{k}\rangle_{(P,Q)}=0, &  \\
\langle[\mathbf{k},\mathbf{i}],e_{i}\rangle_{(P,Q)}+\langle[\mathbf{i},e_{i}],\mathbf{k}\rangle_{(P,Q)}+ \langle[e_{i},\mathbf{k}],\mathbf{i}\rangle_{(P,Q)}=0, &  \\
\langle[\mathbf{i},\mathbf{j}],e_{i}\rangle_{(P,Q)}+\langle[\mathbf{j},e_{i}],\mathbf{i}\rangle_{(P,Q)}+ \langle[e_{i},\mathbf{i}],\mathbf{j}\rangle_{(P,Q)}=0,\quad i=1,2,3. &
  \end{array}
\right. $$
The last three equations are equivalent to
$Q=Q^{T}$ and $\mathrm{tr}(Q)=0$.
This shows that $\langle\cdot,\cdot\rangle_{(P,Q)}$ is cyclic if and only if
$\mathrm{tr}(Q)=\mathrm{tr}(P)=0$, both $P$ and $Q$ are symmetric matrices.
Clearly,  it is non-degenerate if and only if $\det Q\neq0$.
Moreover, by Theorem \ref{2-thm-abc} and Lemma \ref{2-prop-semisimple}, one can prove that every cyclic metric on $\mathfrak{g}=\mathfrak{so}(3,\mathbb{F})+_{\pi}\mathbb{F}^{3}$ must be of the form $\langle\cdot,\cdot\rangle_{(P,Q)}$.
\end{example}

\section{Cyclic quadruples}\label{sec4}
To systematically analyze condition (c) of Theorem \ref{2-thm-abc}, we make the following definition.
\begin{defn}\label{5.1-def-cyclic-quad}
A cyclic quadruple $(\mathfrak{g},\pi,V,\rho)$ over a field $\mathbb{F}$ is a Lie algebra $\mathfrak{g}$ over
$\mathbb{F}$,  a representation $\pi:\mathfrak{g}\rightarrow \mathfrak{gl}(V,\mathbb{F})$ and a linear map $\rho:\mathfrak{g}\rightarrow V^{*}$ satisfying
\begin{eqnarray*}
\rho([x,y])+\pi^{\ast}(x)\rho(y)-\pi^{\ast}(y)\rho(x)=0,\quad \forall x,y\in\mathfrak{g},
\end{eqnarray*}
where $V^{*}$ denotes the dual space of $V$ and $\pi^{*}:\mathfrak{g}\rightarrow \mathfrak{gl}(V^{\ast},\mathbb{F})$ is the dual representation of $\pi$ defined  by
\begin{eqnarray*}
\pi^{*}(x)(f)(v)=-f(\pi(x)(v)), \quad \forall x\in \mathfrak{g}, f\in V^{*}, v\in V.
\end{eqnarray*}
In particular, a cyclic quadruple $(\mathfrak{g},\pi,V,\rho)$ is called irreducible if the representation $\pi$ is irreducible.
Moreover, it is called trivial if $\rho=0$.
\end{defn}

Now let $\mathfrak{g}$ be a  Lie algebra,  $\mathfrak{h}$ be a subalgebra of $\mathfrak{g}$ and $\mathfrak{i}$ be a subspace of $\mathfrak{g}$ satisfying $[\mathfrak{h},\mathfrak{i}]\subset \mathfrak{i}$.
Denote by $\pi: \mathfrak{h}\rightarrow \mathfrak{gl}(\mathfrak{i})$ the representation  given by the adjoint action
$$\pi(h)(i)=[h,i], \quad \forall h\in \mathfrak{h}, i\in \mathfrak{i}.$$
Suppose  that $\mathbf{B}$ is a cyclic metric on $\mathfrak{g}$ and define a linear map
$\rho: \mathfrak{h}\rightarrow \mathfrak{i}^{*}$  by
$$\rho(h)(i)=\mathbf{B}(h,i), \quad \forall h\in \mathfrak{h}, i\in \mathfrak{i}.$$
Note that $\rho=0$ if and only if $\mathbf{B}(\mathfrak{h},\mathfrak{i})=0$.
\begin{prop}\label{5-prop-decom-h-i}
Notation as above. Given a cyclic metric Lie algebra $(\mathfrak{g}, \mathbf{B})$ and a pair $(\mathfrak{h},\mathfrak{i})$,  the associated  quadruple $(\mathfrak{h},\pi,\mathfrak{i},\rho)$ is cyclic.
\end{prop}
\begin{proof}
It follows directly from condition (c) of Theorem \ref{2-thm-abc}.
\end{proof}
\begin{rem}
In the special case where  $\mathfrak{h}=\mathfrak{i}$, the quadruple $(\mathfrak{h}=\mathfrak{i},\pi,\mathfrak{i},\rho)$ is cyclic if and only if the restriction of metric $\mathbf{B}$ on $\mathfrak{i}$ is cyclic.
\end{rem}

Conversely, one can construct cyclic metric Lie algebras from   cyclic quadruples.
Let $(\mathfrak{g},\pi,V,\rho)$ be a  cyclic quadruple, define a new metric Lie algebra
$(L=\mathfrak{g}+_{\pi}V,\mathbf{B}_{L})$ as follows:
\begin{eqnarray*}
\left\{
\begin{aligned}
&[(x,u),(y,v)]_{L}=([x,y]_{\mathfrak{g}},\pi(x)(v)-\pi(y)(u)),\\
&\mathbf{B}_{L}((x,u),(y,v))=\mathbf{B}_{\mathfrak{g}}(x,y)+\rho(x)(v)+\rho(y)(u),
\quad \forall x,y\in \mathfrak{g}, u,v\in V,
\end{aligned}
\right.
\end{eqnarray*}
where $\mathbf{B}_{\mathfrak{g}}$ is any cyclic metric on $\mathfrak{g}$ (could be zero).
 It is clear that $V$ is an Abelian ideal of $L$. By Corollary \ref{2-cor-decom-abelian}, we easily obtain
\begin{thm}\label{5.1-thm-gV}
Notation as above. The metric Lie algebra $(L=\mathfrak{g}+_{\pi}V,\mathbf{B}_{L})$ constructed above is cyclic.
\end{thm}
We now illustrate this construction with a canonical example, which includes Example \ref{3-exam-so(3)}.
\begin{example}\label{4-example-cyclic}
Assume that $\mathfrak{g}$ is a Lie algebra  endowed with a cyclic metric
$\mathbf{B}$.
Let $\pi=\mathrm{ad}: \mathfrak{g}\rightarrow \mathfrak{gl}(\mathfrak{g},\mathbb{F})$
be the adjoint representation of $\mathfrak{g}$ and $\rho:\mathfrak{g}\rightarrow \mathfrak{g}^{*}$
be the linear map uniquely determined by
$$\rho(x)(y)=\mathbf{B}(x,y),\quad x,y\in \mathfrak{g}.$$
Then, for any $x,y,z\in \mathfrak{g}$, one has
\begin{eqnarray*}
   &&  \Big(\rho([x,y])+\pi^{\ast}(x)\rho(y)-\pi^{\ast}(y)\rho(x)\Big)(z)\\
   &=&  \mathbf{B}([x,y],z)-\mathbf{B}(y,[x,z])+\mathbf{B}(x,[y,z])\\
   &=& 0,
\end{eqnarray*}
which asserts that the quadruple $(\mathfrak{g},\pi=\mathrm{ad},V=\mathfrak{g},\rho)$
is cyclic.
Moreover, the resulting cyclic metric Lie algebra $(L=\mathfrak{g}+_{\mathrm{ad}}\mathfrak{g},\mathbf{B}_{L})$ is non-degenerate if and only if the map $\rho:\mathfrak{g}\rightarrow \mathfrak{g}^{*}$ is a linear isometry, if and only if the cyclic metric $\mathbf{B}$ on $\mathfrak{g}$ is non-degenerate.
\end{example}

We now turn to the study of complex cyclic quadruples $(\mathfrak{g},\pi,V,\rho)$ for semisimple $\mathfrak{g}$.
\begin{thm}\label{5-thm-com-trivial}
Let $\mathfrak{g}$ be a complex semisimple Lie algebra containing no simple factor isomorphic to $\mathfrak{sl}(2,\mathbb{C})$.
Then every cyclic quadruple $(\mathfrak{g},\pi,V,\rho)$ is trivial.
\end{thm}
\begin{proof}
Let $\mathfrak{g}$ satisfy the hypothesis, and let
$(\mathfrak{g},\pi,V,\rho)$ be a cyclic quadruple.
Construct the cyclic metric Lie algebra
$(L=\mathfrak{g}+_{\pi}V,\mathbf{B}_{L})$ as in Theorem \ref{5.1-thm-gV}, taking $\mathbf{B}_{\mathfrak{g}}=0$.
Now for every vector $v\in V$, denote by
$$\exp^{\mathrm{ad}\,(0,v)}(\mathfrak{g},0)=\{(x,-\pi(x)(v))|x\in \mathfrak{g}\}$$
 the Lie subalgebra in $L$, which is  isomorphic to $\mathfrak{g}$.
 Thus by Theorem \ref{1-thm-main1}, the restriction of $\mathbf{B}_{L}$ to $\exp^{\mathrm{ad}\,(0,v)}(\mathfrak{g},0)$ is trivial. This means that
\begin{eqnarray*}
  0 &=& \mathbf{B}_{L}(\exp^{\mathrm{ad}\,(0,v)}(x,0),\exp^{\mathrm{ad}\,(0,v)}(y,0)) \\
   &=& \mathbf{B}_{L}((x,-\pi(x)(v)),(y,-\pi(y)(v))) \\
   &=& -\rho(x)(\pi(y)(v))-\rho(y)(\pi(x)(v)) \\
   &=&  (\pi^{*}(y)\rho(x))(v)+(\pi^{*}(x)\rho(y))(v),\quad \forall x,y\in \mathfrak{g}.
\end{eqnarray*}
Therefore $\pi^{*}(y)\rho(x)+\pi^{*}(x)\rho(y)=0$ for all $x,y\in \mathfrak{g}$.
From the defining relation of the cyclic quadruple  $\rho([x,y])+\pi^{\ast}(x)\rho(y)-\pi^{\ast}(y)\rho(x)=0$,
we deduce that
$\rho([x,y])+2\pi^{\ast}(x)\rho(y)=0$, $\forall x,y\in\mathfrak{g}$.

Now let $\mathfrak{g}=\mathfrak{h}+\sum_{\alpha\in \Delta}\mathfrak{g}_{\alpha}$
be a root space decomposition associated  to a Cartan subalgebra $\mathfrak{h}$ of
$\mathfrak{g}$, $\Delta$ is the root system. Obviously, $\pi^{\ast}(h)\rho(h')=0$ holds for all
$h,h'\in \mathfrak{h}$.
Since  $\pi^{\ast}(h)$ acts diagonally on $V^{*}$ for all $h\in \mathfrak{h}$,
it follows that  $\pi^{\ast}(\mathfrak{g})\rho(\mathfrak{h})=0$.
Then we have $\pi^{\ast}(\mathfrak{h})\rho(\mathfrak{g})=0$, which implies that
 $\pi^{\ast}(\mathfrak{g})\rho(\mathfrak{g})=0$.
Therefore  $\rho=0$, completing the proof of the theorem.
\end{proof}
\begin{thm}\label{5-thm-nondegerate-sl(2)}
Let  $(\mathfrak{g},\mathbf{B})$ be a complex  cyclic metric Lie algebra with Levi-decomposition $\mathfrak{g}=\mathfrak{s}+\mathfrak{r}$. If  $\tilde{\mathfrak{s}}\subset \mathfrak{s}$ is a simple factor  not isomorphic to $\mathfrak{sl}(2,\mathbb{C})$, then $\mathbf{B}(\tilde{\mathfrak{s}},\mathfrak{g})=0$.
\end{thm}
\begin{proof}
According to Proposition \ref{5-prop-decom-h-i}, the associated quadruple $(\tilde{\mathfrak{s}},\pi,\mathfrak{r},\rho)$ is cyclic. Then by Theorem \ref{5-thm-com-trivial}, $(\tilde{\mathfrak{s}},\pi,\mathfrak{r},\rho)$ is trivial and so
$\mathbf{B}(\tilde{\mathfrak{s}},\mathfrak{r})=0$.
It follows from Theorem \ref{1-thm-main1} that $\mathbf{B}(\tilde{\mathfrak{s}},\mathfrak{s})=0$.
Hence $\mathbf{B}(\tilde{\mathfrak{s}},\mathfrak{g})=0$, which completes the proof.
\end{proof}
The analogous results over the real field are as follows.
\begin{thm}
Let $\mathfrak{g}$ be a real semisimple Lie algebra with no simple factor isomorphic to $\mathfrak{su}(2)$,
$\mathfrak{sl}(2,\mathbb{R})$ or $\mathfrak{sl}(2,\mathbb{C})$ (regarded as a real Lie algebra).
Then every cyclic quadruple $(\mathfrak{g},\pi,V,\rho)$ is trivial.
\end{thm}
\begin{thm}
Let  $(\mathfrak{g},\mathbf{B})$ be a real  cyclic metric Lie algebra with Levi-decomposition $\mathfrak{g}=\mathfrak{s}+\mathfrak{r}$. If  $\tilde{\mathfrak{s}}\subset \mathfrak{s}$ is a simple factor not isomorphic to $\mathfrak{su}(2)$,
$\mathfrak{sl}(2,\mathbb{R})$ or $\mathfrak{sl}(2,\mathbb{C})$ (regarded as a real Lie algebra),
then $\mathbf{B}(\tilde{\mathfrak{s}},\mathfrak{g})=0$.
\end{thm}

We now focus on complex irreducible cyclic quadruples  $(\mathfrak{g},\pi,V,\rho)$ where $\mathfrak{g}=\mathfrak{sl}(2,\mathbb{C})$.
Denote by $V(k)$ the irreducible $\mathfrak{sl}(2,\mathbb{C})$-module with highest weight $k\in \mathbb{N}$, $\dim V(k)=k+1$. Notation as before, let $\{H,X,Y\}$ be a basis of $\mathfrak{sl}(2,\mathbb{C})$ given by
\begin{eqnarray*}
H=\left(
                  \begin{array}{cc}
                    1 & 0 \\
                    0 & -1 \\
                  \end{array}
                \right),\quad
 X=\left(
            \begin{array}{cc}
              0 & 1 \\
              0 & 0 \\
            \end{array}
          \right),\quad
 Y=\left(
            \begin{array}{cc}
              0 & 0 \\
              1 & 0 \\
            \end{array}
          \right).
\end{eqnarray*}
Let $v_{0}$ be  a maximal vector in $V(k)$,  set  $v_{i}=\frac{1}{i}Y^{i}\cdot v_{0}$ for $0\leq i\leq k$,
and  $v_{-1}=v_{k+1}=0$. The action is given by
\begin{eqnarray*}
\begin{cases}
H\cdot v_{i}=(k-2i)v_{i},\\
Y\cdot v_{i}=(i+1)v_{i+1},\\
X\cdot v_{i}=(k-i+1)v_{i-1},\quad 0\leq i\leq k.
\end{cases}
\end{eqnarray*}
Notice that the action of $\mathfrak{sl}(2,\mathbb{C})$ on  $V(2)$ is equivalent to the adjoint representation of $\mathfrak{sl}(2,\mathbb{C})$.
If  $\mathfrak{g}=\mathfrak{sl}(2,\mathbb{C})\oplus\cdots\oplus\mathfrak{sl}(2,\mathbb{C})$
 is a direct sum of $\mathfrak{sl}(2,\mathbb{C})$ with $m$ copies, $m\geq1$, then the irreducible
 $\mathfrak{g}$-modules must be of the form $V(k_{1})\otimes\cdots\otimes V(k_{m})$, $k_{1},\ldots,k_{m}\in \mathbb{N}$.
 In other words, the irreducible action of  $\mathfrak{g}=\mathfrak{sl}(2,\mathbb{C})\oplus\cdots\oplus\mathfrak{sl}(2,\mathbb{C})$ on $V(k_{1})\otimes\cdots\otimes V(k_{m})$ is given by
 \begin{equation*}
   (x_{1},\ldots,x_{m})\cdot (u_{1}\otimes \cdots\otimes u_{m})
   =(x_{1}\cdot u_{1})\otimes \cdots\otimes u_{m}+\cdots+u_{1}\otimes \cdots\otimes (x_{m}\cdot u_{m}),
 \end{equation*}
where $x_{i}\in \mathfrak{sl}(2,\mathbb{C})$, $u_{i}\in V(k_{i})$, $i=1,\ldots,m$.

To state the main results, we need some technical lemmas.
\begin{lem}\label{5.2-lem1}
Let $(\mathfrak{g},\pi,V,\rho)$ be a complex irreducible cyclic quadruple with $\mathfrak{g}=\mathfrak{sl}(2,\mathbb{C})$. If $\mathrm{dim}\,V\neq 3$, then $\rho=0$.
\end{lem}
\begin{proof}
Since $\pi:\mathfrak{sl}(2,\mathbb{C})\rightarrow \mathfrak{gl}(V,\mathbb{C})$ is irreducible,
its dual representation $\pi^{*}:\mathfrak{sl}(2,\mathbb{C})\rightarrow \mathfrak{gl}(V^{*},\mathbb{C})$ is also  irreducible. The cases $\dim V=1$ and $\dim V=2$ follow directly from Corollary \ref{2-cor-reductive} and Proposition \ref{2-lem-sl(2)}, respectively.
Now suppose that $V^{*}=V(k)$ for some $k\geq 3$, and let  $v_{0}\in V^{*}$ be a maximal vector.
Write the images of the basis elements under $\rho$ as
\begin{eqnarray*}
\rho(H)=\sum_{i=0}^{k}h_{i}v_{i},\quad  \rho(X)=\sum_{i=0}^{k}x_{i}v_{i}, \quad \rho(Y)=\sum_{i=0}^{k}y_{i}v_{i},\quad h_{i},x_{i},y_{i}\in \mathbb{C}.
\end{eqnarray*}
By Definition \ref{5.1-def-cyclic-quad},
\begin{eqnarray*}
\rho([x,y])+\pi^{*}(x)\rho(y)-\pi^{*}(y)\rho(x)=0, \forall x,y\in \mathfrak{sl}(2,\mathbb{C}),
\end{eqnarray*}
we derive the following three equations by evaluating on the pairs $(X,Y)$, $(H,X)$ and $(H,Y)$, respectively.
\begin{eqnarray*}
\begin{cases}
\rho(H)+\pi^{*}(X)\rho(Y)-\pi^{*}(Y)\rho(X)=0, \quad \textcircled{1}\\
2\rho(X)+\pi^{*}(H)\rho(X)-\pi^{*}(X)\rho(H)=0,\quad  \textcircled{2} \\
-2\rho(Y)+\pi^{*}(H)\rho(Y)-\pi^{*}(Y)\rho(H)=0.\quad  \textcircled{3}
\end{cases}
\end{eqnarray*}

A direct computation yields the following system of recurrence relations.
\begin{eqnarray}\label{5-formu-1}
\textcircled{1}&\Leftrightarrow & \sum_{i=0}^{k}h_{i}v_{i}+\sum_{i=0}^{k}(k-i+1)y_{i}v_{i-1}-\sum_{i=0}^{k}(i+1)x_{i}v_{i+1}=0\nonumber\\
               &\Leftrightarrow & \begin{cases}
                                    h_{0}+ky_{1}=0,\\
                                    h_{i}+(k-i)y_{i+1}-ix_{i-1}=0,\quad  1\leq i\leq k-1,\\
                                     h_{k}-kx_{k-1}=0.
                                   \end{cases}
\end{eqnarray}

\begin{eqnarray}\label{5-formu-2}
\textcircled{2}&\Leftrightarrow & 2\sum_{i=0}^{k}x_{i}v_{i}+\sum_{i=0}^{k}(k-2i)x_{i}v_{i}-\sum_{i=0}^{k}(k-i+1)h_{i}v_{i-1}=0\nonumber\\
               &\Leftrightarrow & \begin{cases}
                                     2x_{i}+(k-2i)x_{i}-(k-i)h_{i+1}=0, \quad 0\leq i\leq k-1,\\
                                     2x_{k}+(-k)x_{k}=0.
                                   \end{cases}
\end{eqnarray}

\begin{eqnarray}\label{5-formu-3}
\textcircled{3}&\Leftrightarrow & -2\sum_{i=0}^{k}y_{i}v_{i}+\sum_{i=0}^{k}(k-2i)y_{i}v_{i}-\sum_{i=0}^{k}(i+1)h_{i}v_{i+1}=0\nonumber\\
               &\Leftrightarrow & \begin{cases}
                                     -2y_{0}+ky_{0}=0,\\
                                     -2y_{i}+(k-2i)y_{i}-ih_{i-1}=0, \quad 1\leq i\leq k.
                                     \end{cases}
\end{eqnarray}
From \eqref{5-formu-1}, we have
\begin{eqnarray*}
\begin{cases}
h_{0}=-ky_{1},\quad h_{k}=kx_{k-1},\\
h_{i}=ix_{i-1}-(k-i)y_{i+1}, \quad 1\leq i\leq k-1.
 \end{cases}
\end{eqnarray*}
Substitute into \eqref{5-formu-2} and \eqref{5-formu-3}, we obtain
\begin{eqnarray*}
 \begin{cases}
 (k+2-2i)x_{i}-(k-i)\Big{[}(i+1)x_{i}-(k-i-1)y_{i+2}\Big{]}=0, \quad 0\leq i\leq k-2,\\
 (-k+4)x_{k-1}-kx_{k-1}=0,\\
 x_{k}=0,
 \end{cases}
\end{eqnarray*}
and
\begin{eqnarray*}
\begin{cases}
y_{0}=0,\\
(k-4)y_{1}+ky_{1}=0,\\
(k-2-2i)y_{i}-i\Big{[}(i-1)x_{i-2}-(k-i+1)y_{i}\Big{]}=0, \quad 2\leq i\leq k.
\end{cases}
\end{eqnarray*}
Solving the above equations, one has
$$x_{i}=y_{i}=0, \quad  0\leq i\leq k.$$
Consequently, $h_{i}=0$, $ 0\leq i\leq k$ and thus $\rho=0$.
\end{proof}
\begin{lem}\label{5.2-lem2}
Let $(\mathfrak{g},\pi,V,\rho)$ be a complex irreducible cyclic quadruple with $\mathfrak{g}=\mathfrak{sl}(2,\mathbb{C})\oplus\mathfrak{sl}(2,\mathbb{C})$
and  $V\cong V(2)\otimes V(2)$.  Then $\rho=0$.
\end{lem}
\begin{proof}
Under the assumption, $\pi^{\ast}$ is isomorphic to $\pi$.
We may identify $V^{\ast}$ with $\mathfrak{sl}(2,\mathbb{C})\otimes\mathfrak{sl}(2,\mathbb{C})$, equipped with the representation $\pi^{\ast}=\mathrm{ad}\otimes \mathrm{ad}$, defined by
$$\pi^{\ast}(x,y)(u\otimes v)=[x,u]\otimes v+u\otimes [y,v],$$
$\forall (x,y)\in \mathfrak{sl}(2,\mathbb{C})\oplus\mathfrak{sl}(2,\mathbb{C})$, $u\otimes v\in \mathfrak{sl}(2,\mathbb{C})\otimes \mathfrak{sl}(2,\mathbb{C})$.
We proceed in seven steps to show that $\rho=0$.

Step 1. Noting that
\begin{eqnarray*}
\pi^{\ast}((X,0))\rho((0,H))=\pi^{\ast}((0,H))\rho((X,0))\in \mathfrak{sl}(2,\mathbb{C}) \otimes X+\mathfrak{sl}(2,\mathbb{C})\otimes Y,
\end{eqnarray*}
we obtain
$$\rho((0,H))\in \mathfrak{sl}(2,\mathbb{C}) \otimes X+\mathfrak{sl}(2,\mathbb{C})\otimes Y,$$
and similarly
$$\rho((H,0))\in X\otimes\mathfrak{sl}(2,\mathbb{C}) +Y\otimes\mathfrak{sl}(2,\mathbb{C}).$$

Step 2. Noting that
\begin{eqnarray*}
2\rho((X,0))+\pi^{\ast}((H,0))\rho((X,0))=\pi^{\ast}((X,0))\rho((H,0))\in H\otimes \mathfrak{sl}(2,\mathbb{C}),
\end{eqnarray*}
we obtain
\begin{eqnarray*}
\rho((X,0))\in Y\otimes\mathfrak{sl}(2,\mathbb{C}) +H\otimes\mathfrak{sl}(2,\mathbb{C}),
\end{eqnarray*}
and similarly
\begin{eqnarray*}
\begin{cases}
\rho((Y,0))\in X\otimes\mathfrak{sl}(2,\mathbb{C}) +H\otimes\mathfrak{sl}(2,\mathbb{C}),\\
\rho((0,X))\in \mathfrak{sl}(2,\mathbb{C}) \otimes Y+\mathfrak{sl}(2,\mathbb{C})\otimes H,\\
\rho((0,Y))\in \mathfrak{sl}(2,\mathbb{C}) \otimes X+\mathfrak{sl}(2,\mathbb{C})\otimes H.
\end{cases}
\end{eqnarray*}

Step 3. By  comparing terms from  step 1 and step 2, we have
\begin{eqnarray*}
\pi^{\ast}((X,0))\rho((0,H))\in \mathbb{C} X\otimes X+\mathbb{C} H\otimes X+\mathbb{C} X\otimes Y+\mathbb{C} H\otimes Y
\end{eqnarray*}
and
\begin{eqnarray*}
\pi^{\ast}((0,H))\rho((X,0))\in \mathbb{C} Y\otimes X+\mathbb{C} Y\otimes Y+\mathbb{C} H\otimes X+\mathbb{C} H\otimes Y.
\end{eqnarray*}
As $\pi^{\ast}((X,0))\rho((0,H))=\pi^{\ast}((0,H))\rho((X,0))$, one has
\begin{eqnarray*}
\rho((0,H))\in \mathbb{C} X\otimes X+\mathbb{C} Y\otimes X+\mathbb{C} X\otimes Y+\mathbb{C} Y\otimes Y,
\end{eqnarray*}
and
\begin{eqnarray*}
\rho((X,0))\in \mathbb{C}Y\otimes H+H\otimes\mathfrak{sl}(2,\mathbb{C}).
\end{eqnarray*}
Similarly, we obtain
\begin{eqnarray*}
\rho((H,0))\in \mathbb{C} X\otimes X+\mathbb{C} Y\otimes X+\mathbb{C} X\otimes Y+\mathbb{C} Y\otimes Y,
\end{eqnarray*}
and
\begin{eqnarray*}
\rho((0,X))\in  \mathbb{C} H\otimes Y+\mathfrak{sl}(2,\mathbb{C})\otimes H.
\end{eqnarray*}

Step 4. By step 3,
\begin{eqnarray*}
2\rho((X,0))+\pi^{\ast}((H,0))\rho((X,0))=\pi^{\ast}((X,0))\rho((H,0))\in \mathbb{C}H\otimes X+\mathbb{C}H\otimes Y,
\end{eqnarray*}
we obtain
$$\rho((X,0))\in \mathbb{C}Y\otimes H+\mathbb{C}H\otimes X+\mathbb{C}H\otimes Y.$$
Similarly,
\begin{eqnarray*}
\begin{cases}
\rho((Y,0))\in \mathbb{C}X\otimes H+\mathbb{C}H\otimes X+\mathbb{C}H\otimes Y,\\
\rho((0,X))\in \mathbb{C}H\otimes Y+\mathbb{C}X\otimes H+\mathbb{C}Y\otimes H, \\
\rho((0,Y))\in \mathbb{C}H\otimes X+\mathbb{C}X\otimes H+\mathbb{C}Y\otimes H.
\end{cases}
\end{eqnarray*}

Step 5. By step 4, one has
\begin{eqnarray*}
\begin{cases}
\pi^{\ast}((0,Y))\rho((X,0))\in \mathbb{C}Y\otimes Y+\mathbb{C}H\otimes H,\\
\pi^{\ast}((X,0))\rho((0,Y))\in \mathbb{C}X\otimes X+\mathbb{C}H\otimes H.
\end{cases}
\end{eqnarray*}
As $\pi^{\ast}((0,Y))\rho((X,0))=\pi^{\ast}((X,0))\rho((0,Y))$, we obtain
\begin{eqnarray*}
\begin{cases}
\rho((X,0))\in \mathbb{C} H\otimes X+\mathbb{C} H\otimes Y,\\
\rho((0,Y))\in \mathbb{C} X\otimes H+\mathbb{C} Y\otimes H.
\end{cases}
\end{eqnarray*}
Similarly,
\begin{eqnarray*}
\begin{cases}
\rho((0,X))\in \mathbb{C} X\otimes H+\mathbb{C} Y\otimes H,\\
\rho((Y,0))\in \mathbb{C} H\otimes X+\mathbb{C} H\otimes Y.
\end{cases}
\end{eqnarray*}

Step 6. By step 5, assume that
\begin{eqnarray*}
\begin{cases}
\rho((X,0))=a H\otimes X+b H\otimes Y,\\
\rho((Y,0))= c H\otimes X+d H\otimes Y, \quad  a,b,c,d\in \mathbb{C}.
\end{cases}
\end{eqnarray*}
We obtain
\begin{eqnarray*}
\rho((H,0))&=&\rho([(X,0),(Y,0)])\\
           &=&-\pi^{\ast}((X,0))\rho((Y,0))+\pi^{\ast}((Y,0))\rho((X,0))\\
           &=& 2cX\otimes X+2dX\otimes Y-2aY\otimes X-2bY\otimes Y.
\end{eqnarray*}
Consequently,
\begin{eqnarray*}
0 &=& \rho([(H,0),(X,0)])+\pi^{\ast}((H,0))\rho((X,0))-\pi^{\ast}((X,0))\rho((H,0))\\
  &=& 2\rho((X,0))+2aH\otimes X+2bH\otimes Y\\
  &=& 4aH\otimes X+4bH\otimes Y,
\end{eqnarray*}
which implies that $a=b=0$. Similarly, we have $\rho((0,X))=\rho((X,0))=0$.

Step 7. Finally, by step 6,
\begin{eqnarray*}
0 &=& \pi^{\ast}((Y,0))\rho((0,X))=\pi^{\ast}((0,X))\rho((Y,0))\\
  &=& d H\otimes H,
\end{eqnarray*}
one has $d=0$. Thus $\rho((H,0))=2c X\otimes X$. As
\begin{eqnarray*}
\pi^{\ast}((H,0))\rho((0,H))=\pi^{\ast}((0,H))\rho((H,0))=4cX\otimes X,
\end{eqnarray*}
we have $\rho((0,H))=2cX\otimes X$.

Now
\begin{eqnarray*}
0 &=& \pi^{\ast}((0,H))\rho((Y,0))-\pi^{\ast}((Y,0))\rho((0,H))\\
  &=& 2cH\otimes X+2cH\otimes X\\
  &=& 4cH\otimes X,
\end{eqnarray*}
one has $c=0$ and $\rho((Y,0))=\rho((H,0))=0$. Similarly, we have $\rho((0,Y))=\rho((0,H))=0$.
Therefore $\rho=0$, which completes the proof of the lemma.
\end{proof}
\begin{lem}\label{5.2-lem3}
Let $(\mathfrak{g},\pi,V,\rho)$ be a complex irreducible cyclic quadruple with $\mathfrak{g}=\mathfrak{sl}(2,\mathbb{C})\oplus\mathfrak{sl}(2,\mathbb{C})$.
If $\dim V\neq 3$, then $\rho=0$.
\end{lem}
\begin{proof}
Suppose that $V=V(k_{1})\otimes V(k_{2})$ for some $k_{1},k_{2}\in \mathbb{N}$.

Case 1. $k_{1}\neq2$ and $k_{2}\neq2$. It follows from Lemma \ref{5.2-lem1}.

Case 2. $k_{1}=k_{2}=2$. It follows from Lemma \ref{5.2-lem2}.

Case 3. Either $k_{1}=2$, $k_{2}\neq2$ or $k_{1}\neq2$, $k_{2}=2$.
Without losing generality, we may assume that $V=V(2)\otimes V(k)$ with $k\neq2$.
Consider the cyclic metric Lie algebra  $(L=\mathfrak{g}+_{\pi}V,\mathbf{B}_{L})$ constructed as
in Theorem \ref{5.1-thm-gV},
where  $\mathbf{B}_{L}((x,u),(y,v))=\rho(x)(v)+\rho(y)(u)$, $\forall x,y\in \mathfrak{g}$, $u,v\in V$.
By Lemma \ref{5.2-lem1}, we have
$\mathbf{B}_{L}((0,\mathfrak{sl}(2,\mathbb{C})),V)=0$.
Then for all $x,y,z\in \mathfrak{sl}(2,\mathbb{C})$, $u\in V(2)$ and $v\in V(k)$, we obtain
\begin{eqnarray*}
  0 &=& \mathbf{B}_{L}([(x,y),(x,z)]_{L},u\otimes v) \\
   &=&  -\mathbf{B}_{L}([(x,z),u\otimes v]_{L},(x,y))-\mathbf{B}_{L}([u\otimes v,(x,y)]_{L},(x,z))\\
   &=& -\mathbf{B}_{L}((x\cdot u)\otimes v+u\otimes (z\cdot v),(x,y))
   +\mathbf{B}_{L}((x\cdot u)\otimes v+u\otimes (y\cdot v),(x,z)) \\
   &=& -\mathbf{B}_{L}((x\cdot u)\otimes v+u\otimes (z\cdot v),(x,0))
   +\mathbf{B}_{L}((x\cdot u)\otimes v+u\otimes (y\cdot v),(x,0)) \\
   &=&  \mathbf{B}_{L}( u\otimes (y\cdot v-z\cdot v),(x,0)).
\end{eqnarray*}
This asserts that $\mathbf{B}_{L}((\mathfrak{sl}(2,\mathbb{C}),0),V)=0$ and consequently
$\mathbf{B}_{L}(\mathfrak{g},V)=0$. Hence  $\rho=0$.
\end{proof}
Combining Lemmas \ref{5.2-lem1}, \ref{5.2-lem2} and \ref{5.2-lem3}, we easily obtain
\begin{thm}\label{5.2-thm1}
Let $(\mathfrak{g},\pi,V,\rho)$ be a complex irreducible cyclic quadruple with
  $\mathfrak{g}=\mathfrak{sl}(2,\mathbb{C})\oplus\cdots\oplus\mathfrak{sl}(2,\mathbb{C})$
  ($m\geq1$ copies).
If $\dim V\neq 3$, then $\rho=0$.
\end{thm}
\begin{proof}
This follows by induction on $m$, using the complete reducibility of representations of semisimple Lie algebras and Lemmas \ref{5.2-lem1}, \ref{5.2-lem2} and \ref{5.2-lem3}. The base cases $m=1,2$ are covered.
For $m>2$, any irreducible representation is a tensor product of irreducible representations of the factors, and its dimension is the product of the dimensions. If $\dim V\neq 3$, we can decompose $\mathfrak{g}=\mathfrak{g}_{1}\oplus \mathfrak{g}_{2}$ with $V=V_{1}\otimes V_{2}$, where $\mathfrak{g}_{i}$ acts irreducibly on $V_{i}$ and $\dim V_{i}\neq 3$ for $i=1,2$.
\end{proof}
\begin{thm}\label{4-thm-b}
Let  $(\mathfrak{g},\mathbf{B})$ be a complex  cyclic metric Lie algebra with Levi-decomposition $\mathfrak{g}=\mathfrak{s}+\mathfrak{r}$. Suppose that $\tilde{\mathfrak{s}}=\mathfrak{sl}(2,\mathbb{C})\oplus\cdots\oplus\mathfrak{sl}(2,\mathbb{C})\subset \mathfrak{s}$ is a semisimple factor,
 and the dimension of every $\tilde{\mathfrak{s}}$-irreducible subspace of  $\mathfrak{r}$ is not equal to $3$,  then $\mathbf{B}(\tilde{\mathfrak{s}},\mathfrak{r})=0$.
\end{thm}
\begin{proof}
For any $\tilde{\mathfrak{s}}$-irreducible subspace $W\subset \mathfrak{r}$, we have $\dim W\neq 3$ by hypothesis. By Proposition \ref{5-prop-decom-h-i}, the associated quadruple $(\tilde{\mathfrak{s}},\pi,W,\rho)$
is cyclic.
Theorem \ref{5.2-thm1} implies that this quadruple is trivial, so $\mathbf{B}(\tilde{\mathfrak{s}},W)=0$.
Since $\mathfrak{r}$ is a direct sum of such irreducible subspaces, we conclude $\mathbf{B}(\tilde{\mathfrak{s}},\mathfrak{r})=0$.
\end{proof}
\begin{thm}\label{4-thm-c}
Let  $(\mathfrak{g},\mathbf{B})$ be a   cyclic metric Lie algebra over the complex or real field,
 then $\mathbf{B}(\mathfrak{s},[\mathfrak{a},\mathrm{nil}\,\mathfrak{g}])
=\mathbf{B}(\mathfrak{a},[\mathfrak{s},\mathrm{nil}\,\mathfrak{g}])=0$.
\end{thm}
\begin{proof}
It suffices to prove the result over $\mathbb{C}$.
For $x\in \mathfrak{a}$ and $\tilde{\mathfrak{s}}\subset \mathfrak{s}$ a simple factor, let $V\subset \mathrm{nil}\,\mathfrak{g}$ be an irreducible space of $\tilde{\mathfrak{s}}$, then $[x,V]$ is also an
irreducible space of $\tilde{\mathfrak{s}}$ isomorphic to $V$, but not necessarily equal to $V$.

If $\dim V\neq 3$ or $\tilde{\mathfrak{s}}\neq\mathfrak{sl}(2,\mathbb{C})$, then by Theorem \ref{5-thm-nondegerate-sl(2)} and Theorem \ref{4-thm-b}, we have
$\mathbf{B}(\tilde{\mathfrak{s}}, V)=\mathbf{B}(\tilde{\mathfrak{s}}, [x,V])=0$. Therefore,
\begin{eqnarray*}
0 &=& \mathbf{B}(x,[\tilde{\mathfrak{s}},V])+\mathbf{B}(\tilde{\mathfrak{s}},[V,x])
+\mathbf{B}( V,[x,\tilde{\mathfrak{s}}])\\
  &=& \mathbf{B}(x,[\tilde{\mathfrak{s}},V]).
\end{eqnarray*}

If $\dim V=3$ and $\tilde{\mathfrak{s}}=\mathfrak{sl}(2,\mathbb{C})$, then the action  is equivalent to the adjoint representation. We can identify $\tilde{\mathfrak{s}}+V$ with $\mathfrak{so}(3,\mathbb{C})+_{\pi}\mathbb{C}^{3}$ as in Example \ref{3-exam-so(3)}, and   the symmetry of $Q$ is also preserved for $\tilde{\mathfrak{s}}+[x,V]$.
For instance, from $\mathbf{B}(\mathbf{i},e_{2})=\mathbf{B}(\mathbf{j},e_{1})$,
 it follows that $\mathbf{B}(\mathbf{i},[x,e_{2}])=\mathbf{B}(\mathbf{j},[x,e_{1}])$,
which implies
$\mathbf{B}(x,[\mathbf{i},e_{2}])=\mathbf{B}(x,[\mathbf{j},e_{1}])$.
Consequently,  $\mathbf{B}(x,e_{3})=0$. A similar argument shows that $\mathbf{B}(x,V)=0$.

In summary, we conclude that $\mathbf{B}(x,[\tilde{\mathfrak{s}},V])=0$.
Hence  $\mathbf{B}(x,[\tilde{\mathfrak{s}},\mathrm{nil}\,\mathfrak{g}])=0$, and therefore
$\mathbf{B}(\mathfrak{a},[\mathfrak{s},\mathrm{nil}\,\mathfrak{g}])=0$. This completes the proof.
\end{proof}

\textbf{Proof of Theorem \ref{1-thm-main2}}.
This is a direct consequence of Lemma \ref{2-prop-semisimple}, Theorems \ref{5-thm-nondegerate-sl(2)},
\ref{4-thm-b} and \ref{4-thm-c}.
\qed

\section{Proof of Theorem \ref{1-thm-main3}}\label{sec5}
In this section, we present the proof of Theorem \ref{1-thm-main3}. To accomplish this, we need a technical proposition.
\begin{prop}\label{6-prop-a}
Let $\mathfrak{g}$ be a Lie algebra over $\mathbb{R}$ or $\mathbb{C}$,
$\mathfrak{g}=\mathfrak{s}+\mathfrak{r}=\mathfrak{s}+\mathfrak{a}+\mathrm{nil}\,\mathfrak{g}$ be a Levi-decomposition of $\mathfrak{g}$.
Let $0=C^{0}(\mathrm{nil}\,\mathfrak{g}) \subset C^{1}(\mathrm{nil}\,\mathfrak{g})=C(\mathrm{nil}\,\mathfrak{g})\subset C^{2}(\mathrm{nil}\,\mathfrak{g})\subset\cdots\subset\mathrm{nil}\,\mathfrak{g}$ be the upper central series of $\mathrm{nil}\,\mathfrak{g}$, and $[\mathfrak{s},C^{k-1}(\mathrm{nil}\,\mathfrak{g})]=0$, but $[\mathfrak{s},C^{k}(\mathrm{nil}\,\mathfrak{g})]\neq0$ for a number $k\in \mathbb{N}$. Suppose  that $\mathbf{B}$ is a cyclic metric on $\mathfrak{g}$,
then $\mathbf{B}([\mathfrak{s},C^{k}(\mathrm{nil}\,\mathfrak{g})],\mathfrak{r})=0$ and
$\mathbf{B}([\mathfrak{a},[\mathfrak{s},C^{k}(\mathrm{nil}\,\mathfrak{g})]],\mathfrak{g})=0$.
\end{prop}
\begin{proof}
It suffices to prove the result over $\mathbb{C}$. Under the action of $\mathfrak{s}$ on $\mathrm{nil}\,\mathfrak{g}$, we have the vector space decomposition   $\mathrm{nil}\,\mathfrak{g}=\mathfrak{n}_{0}+[\mathfrak{s},\mathrm{nil}\,\mathfrak{g}]$, where $\mathfrak{n}_{0}$ is a nilpotent subalgebra of $\mathrm{nil}\,\mathfrak{g}$ satisfying $[\mathfrak{s},\mathfrak{n}_{0}]=0$. By Schur's Lemma, we easily obtain $[\mathfrak{n}_{0},[\mathfrak{s},\mathrm{nil}\,\mathfrak{g}]]=0$, which implies that $\mathfrak{n}_{0}$ is an ideal of $\mathfrak{g}$.
 Then we have
 $$ \mathbf{B}([\mathfrak{s},\mathrm{nil}\,\mathfrak{g}],\mathfrak{n}_{0})
 =\mathbf{B}([\mathfrak{s},[\mathfrak{s},\mathrm{nil}\,\mathfrak{g}]],\mathfrak{n}_{0})
 =-\mathbf{B}([[\mathfrak{s},\mathrm{nil}\,\mathfrak{g}],\mathfrak{n}_{0}],\mathfrak{s})
 -\mathbf{B}([\mathfrak{n}_{0},\mathfrak{s}],[\mathfrak{s},\mathrm{nil}\,\mathfrak{g}])
 =0.$$
  Thus, by Lemma \ref{2-lem-decomposition} and Theorem \ref{4-thm-c}, we obtain that
\begin{eqnarray*}
\mathbf{B}(\mathfrak{s}+[\mathfrak{s},\mathrm{nil}\,\mathfrak{g}],\mathfrak{a}+\mathfrak{n}_{0})=0.
\end{eqnarray*}

Since $\mathfrak{s}$ acts completely reducibly on $\mathrm{nil}\,\mathfrak{g}$, there exists an $\mathfrak{s}$-invariant subspace decomposition $\mathrm{nil}\,\mathfrak{g}=C^{k}(\mathrm{nil}\,\mathfrak{g})+W$, where $[\mathfrak{s},W]\subset W$.  By definition,
 $[\mathfrak{s},[C^{k}(\mathrm{nil}\,\mathfrak{g}),\mathrm{nil}\,\mathfrak{g}]]=0$, which implies
\begin{equation*}
  0=\mathbf{B}\left(\mathfrak{s},[C^{k}(\mathrm{nil}\,\mathfrak{g}),\mathrm{nil}\,\mathfrak{g}]\right)
=-\mathbf{B}\left(C^{k}(\mathrm{nil}\,\mathfrak{g}),[\mathrm{nil}\,\mathfrak{g},\mathfrak{s}]\right)
-\mathbf{B}\left(\mathrm{nil}\,\mathfrak{g},[\mathfrak{s},C^{k}(\mathrm{nil}\,\mathfrak{g})]\right).
\end{equation*}
Thus by  Corollary \ref{2-cor-decom-abelian}, we  can define a new cyclic metric Lie algebra
$\left((\mathfrak{s}+\mathrm{nil}\,\mathfrak{g},[\cdot,\cdot]'),\widetilde{\mathbf{B}}\right)$ as follows:
\begin{equation*}
  \begin{cases}
[\mathfrak{s},\mathfrak{s}]'=[\mathfrak{s},\mathfrak{s}],
\quad [\mathfrak{s},\mathrm{nil}\,\mathfrak{g}]'=[\mathfrak{s},\mathrm{nil}\,\mathfrak{g}],\quad [\mathrm{nil}\,\mathfrak{g},\mathrm{nil}\,\mathfrak{g}]'=0,\\
\widetilde{\mathbf{B}}(\mathfrak{s},\mathfrak{s}+\mathrm{nil}\,\mathfrak{g})
=\widetilde{\mathbf{B}}(C^{k}(\mathrm{nil}\,\mathfrak{g}),C^{k}(\mathrm{nil}\,\mathfrak{g}))
=\widetilde{\mathbf{B}}(W,W)=0,\\
\widetilde{\mathbf{B}}(C^{k}(\mathrm{nil}\,\mathfrak{g}),W)=\mathbf{B}(C^{k}(\mathrm{nil}\,\mathfrak{g}),W).
\end{cases}
\end{equation*}
By Corollary \ref{2-cor-abelian},
we obtain that $\widetilde{\mathbf{B}}([\mathfrak{s},\mathrm{nil}\,\mathfrak{g}]',\mathrm{nil}\,\mathfrak{g})=0$.
Thus $\mathbf{B}([\mathfrak{s},C^{k}(\mathrm{nil}\,\mathfrak{g})],W)=0$.
Moreover, since for each $x\in \mathfrak{s}$, the operator $\mathrm{ad}\,x|_{C^{k}(\mathrm{nil}\,\mathfrak{g})}$ is symmetric,
Lemma \ref{2-prop-semisimple} implies that
 $\mathbf{B}([\mathfrak{s},C^{k}(\mathrm{nil}\,\mathfrak{g})],C^{k}(\mathrm{nil}\,\mathfrak{g}))=0$.
It follows  that
\begin{equation*}
 \mathbf{B}([\mathfrak{s},C^{k}(\mathrm{nil}\,\mathfrak{g})],\mathfrak{r})
=\mathbf{B}([\mathfrak{s},C^{k}(\mathrm{nil}\,\mathfrak{g})],\mathrm{nil}\,\mathfrak{g})=0.
\end{equation*}

We now suppose  that $[\mathfrak{a},[\mathfrak{s},C^{k}(\mathrm{nil}\,\mathfrak{g})]]\neq0$. Then under the action of $\mathfrak{a}$ on $[\mathfrak{s},\mathrm{nil}\,\mathfrak{g}]$, we have a vector space decomposition
$[\mathfrak{s},\mathrm{nil}\,\mathfrak{g}]=\sum_{\lambda\in \mathfrak{a}^{*}}V_{\lambda}$, where
$$V_{\lambda}=\{v\in[\mathfrak{s},\mathrm{nil}\,\mathfrak{g}]|[x,v]=\lambda(x)v,\forall x\in \mathfrak{a} \}.$$
Clearly,
$[\mathfrak{a},[\mathfrak{s},C^{k}(\mathrm{nil}\,\mathfrak{g})]]
=\sum_{\lambda\neq0}\left(V_{\lambda}\cap C^{k}(\mathrm{nil}\,\mathfrak{g})\right)$,
and $[\mathfrak{s},V_{\lambda}]=V_{\lambda}$.
Since
\begin{equation*}
  \mathbf{B}(\mathfrak{s},[\mathfrak{a},V_{\lambda}])
=-\mathbf{B}(\mathfrak{a},[V_{\lambda},\mathfrak{s}])-\mathbf{B}(V_{\lambda},[\mathfrak{s},\mathfrak{a}])=0,
\end{equation*}
we  conclude that $\mathbf{B}(\mathfrak{s},V_{\lambda})=0$ for all $\lambda\neq0$.
In summary, we conclude that
$\mathbf{B}(V_{\lambda}\cap C^{k}(\mathrm{nil}\,\mathfrak{g}),
\mathfrak{s}+\mathfrak{r})=0$ for all $\lambda\neq0$. Therefore $\mathbf{B}([\mathfrak{a},[\mathfrak{s},C^{k}(\mathrm{nil}\,\mathfrak{g})]],\mathfrak{g})=0$.
\end{proof}

\textbf{Proof of Theorem \ref{1-thm-main3}}.
Let $(\mathfrak{g},\mathbf{B})$ be  a non-degenerate cyclic metric Lie algebra and $\mathfrak{g}=\mathfrak{s}+\mathfrak{r}$ be a Levi-decomposition of $\mathfrak{g}$, where $\mathfrak{r}=\mathfrak{a}+\mathrm{nil}\,\mathfrak{g}$.
 Take the notation as in Proposition \ref{6-prop-a},  we have
$\mathbf{B}(\mathfrak{s}+[\mathfrak{s},\mathrm{nil}\,\mathfrak{g}],\mathfrak{a}+\mathfrak{n}_{0})=0.$

(a) It suffices to prove the result over $\mathbb{C}$. Suppose conversely that $[\mathfrak{s},\mathrm{nil}\,\mathfrak{g}]\neq 0$. Then there exists an integer $k\geq2$ such that
$[\mathfrak{s},C^{k-1}(\mathrm{nil}\,\mathfrak{g})]=0$ and $[\mathfrak{s},C^{k}(\mathrm{nil}\,\mathfrak{g})]\neq0$.
 One can then choose $\mathfrak{s}$-irreducible subspaces $V\subset [\mathfrak{s},C^{k}(\mathrm{nil}\,\mathfrak{g})]$ and $U\subset \mathrm{nil}\,\mathfrak{g}$ such that
$[V,U]\neq0$. It follows from Proposition \ref{6-prop-a} that $[\mathfrak{a},V]=0$, and hence
\begin{equation*}
  \mathbf{B}(\mathfrak{a},[V,U])=-\mathbf{B}(V,[U,\mathfrak{a}])-\mathbf{B}(U,[\mathfrak{a},V])=0.
\end{equation*}
Furthermore, we observe that
\begin{equation*}
 \mathbf{B}(\mathfrak{n}_{0},[V,U])=-\mathbf{B}(V,[U,\mathfrak{n}_{0}])-\mathbf{B}(U,[\mathfrak{n}_{0},V])=0.
\end{equation*}
This implies that $\mathbf{B}([V,U],\mathfrak{g})=0$, which contradicts the non-degeneracy of $\mathbf{B}$.

(b) Suppose that $[\mathfrak{s},C(\mathrm{nil}\,\mathfrak{g})]\neq0$. It follows from Proposition \ref{6-prop-a}  that $[\mathfrak{a},[\mathfrak{s},C(\mathrm{nil}\,\mathfrak{g})]]=0$,
hence $[\mathfrak{s},C(\mathrm{nil}\,\mathfrak{g})]\subset C(\mathfrak{r})$, the center of radical $\mathfrak{r}$.
Moreover, we have $\mathbf{B}([\mathfrak{s},C(\mathrm{nil}\,\mathfrak{g})],\mathfrak{r})=0$.

It follows from Theorem \ref{5-thm-nondegerate-sl(2)} that every simple factor of $\mathfrak{s}$ is isomorphic to $\mathfrak{sl}(2,\mathbb{C})$ over $\mathbb{C}$, and  to $\mathfrak{su}(2)$, $\mathfrak{sl}(2,\mathbb{R})$, or $\mathfrak{sl}(2,\mathbb{C})$ (regarded as a real Lie algebra) over $\mathbb{R}$.
By Theorem \ref{5.2-thm1} and its real analogue, there exist a direct sum decomposition $\mathfrak{s}=\mathfrak{s}_{1}\oplus\mathfrak{s}_{2}$ and an $\mathfrak{s}$-invariant subspace
$W\subset [\mathfrak{s},C(\mathrm{nil}\,\mathfrak{g})]$, such that $\mathfrak{s}_{1}$ is simple,  the action of  $\mathfrak{s}_{1}$  on  $W$ is equivalent to the adjoint representation of $\mathfrak{s}_{1}$, and $[\mathfrak{s}_{2},W]=0$.
This implies that the restriction of $\mathbf{B}$ to $\mathfrak{s}_{1}+W$ is non-degenerate,
 and $\mathbf{B}(\mathfrak{s}_{1},W)$ is induced by a non-degenerate cyclic metric on $\mathfrak{s}_{1}$ as described in Example \ref{4-example-cyclic}.
 Let $\mathfrak{h}\subset \mathfrak{g}$ be the orthogonal complement of $\mathfrak{s}_{1}+W$ in $\mathfrak{g}$. Then there exist a Lie bracket $[\cdot,\cdot]_{\mathfrak{h}}$ on $\mathfrak{h}$ and a cocycle $\theta\in (\wedge^{2}\mathfrak{h}^{\ast})\otimes W$ such that
\begin{eqnarray*}
[x,y]=[x,y]_{\mathfrak{h}}+\theta(x,y), \forall x,y\in\mathfrak{h}.
\end{eqnarray*}
Clearly, $\mathfrak{s}_{2}\subset \mathfrak{h}$, $[\mathfrak{s}_{1},\mathfrak{h}]\subset \mathfrak{h}$
and $\left((\mathfrak{h},[\cdot,\cdot]_{\mathfrak{h}}),\mathbf{B}|_{\mathfrak{h}}\right)$ is cyclic.
This shows that $(\mathfrak{g},\mathbf{B})$ is a double extension of the non-degenerate  cyclic metric Lie algebra $\left((\mathfrak{h},[\cdot,\cdot]_{\mathfrak{h}}),\mathbf{B}|_{\mathfrak{h}}\right)$ with respect to $\mathfrak{s}_{1}$.\qed



\begin{thebibliography}{99}

\bibitem{abol25} F. Abid, S. Benayadi, M. Boucetta, H. Ouali, H. Lebzioui,
Cyclic Riemannian Lie groups: description and curvatures, http://arxiv.org/abs/2504.16759v1.

\bibitem{as58} W. Ambrose, I.M. Singer, On homogeneous Riemannian manifolds, Duke Math. J. 25 (1958), 647--669.

\bibitem{bie97}  L. Bieszk, Classification of five-dimensional Lie algebras of class $T_{2}$, Demonstratio Math. 30 (2) (1997), 403--424.


\bibitem{cl16} G. Calvaruso, M.C. L\'{o}pez, Cyclic Lorentzian Lie groups, Geom. Dedicata, 181 (2016), 119--136.

\bibitem{cl19} G. Calvaruso, M.C. L\'{o}pez, Pseudo-Riemannian Homogeneous Structures,
 Developments in Mathematics 59, Springer, Cham, 2019.







\bibitem{deleo14} B. de Leo, Homogeneous pseudo-Riemannian structures of class $\mathcal{T}_{2}$ on low-dimensional generalized symmetric spaces, Balkan J. Geom. Appl. 19 (2) (2014), 73--85.

\bibitem{deleoM08} B. de Leo,  R.A. Marinosci, Special homogeneous structures on pseudo-Riemannian manifolds, JP J. Geom. Topol. 8 (3) (2008), 203--228.

\bibitem{Dzhum09} A.S. Dzhumadil'daev, Algebras with skew-symmetric identity of degree 3, J. Math. Sci. 161 (2009), 11--30.

\bibitem{DB09} A.S. Dzhumadil'daev, A.B.  Bakirova, Simple two-sided anti-Lie-admissible algebras, J. Math. Sci. 161
(2009) 31--36.

\bibitem{Dzhum10} A.S. Dzhumadil'daev, P. Zusmanovich, Commutative 2-cocycles on Lie algebras, J. Algebra, 324 (2010), 732--748.

\bibitem{fp06}  M. Falcitelli, A.M. Pastore, $f$-structures of Kenmotsu type,
 Mediterr. J. Math. 3 (3-4) (2006), 549--564.

\bibitem{fs87} G. Favre, L.J. Santharoubane, Symmetric, invariant, non-degenerate bilinear form on a Lie algebra, J. Algebra 105 (1987), 451--464.


\bibitem{ggo15} P.M. Gadea, J.C. Gonz\'{a}lez-D\'{a}vila, J.A. Oubi\~{n}a, Cyclic metric Lie groups,
Monatsh. Math. 176 (2) (2015), 219--239.

\bibitem{ggo16} P.M. Gadea, J.C. Gonz\'{a}lez-D\'{a}vila, J.A. Oubi\~{n}a, Cyclic homogeneous Riemannian manifolds, Ann. Mat. Pura Appl. (4) 195 (5) (2016),  1619--1637.

\bibitem{go97} P.M. Gadea, J.A. Oubi\~{n}a, Reductive homogeneous pseudo-Riemannian manifolds,
 Monatsh. Math. 124 (1) (1997), 17--34.

\bibitem{kt87} O. Kowalski, F. Tricerri, Riemannian manifolds of dimension $n \leq 4$ admitting a homogeneous structure of class $T_{2}$,  Confer. Sem. Mat. Univ. Bari, 222 (1987), 24 pp.

\bibitem{LB22} G. Liu, C. Bai, Anti-pre-Lie algebras, Novikov algebras and commutative 2-cocycles on Lie algebras, J. Algebra. 609 (2022), 337--379.

\bibitem{LB24} G. Liu, C. Bai, A bialgebra theory for transposed Poisson algebras via anti-pre-Lie bialgebras and anti-pre-Lie Poisson bialgebras. Commun. Contemp. Math. 26 (8) (2024), Paper No. 2350050, 49 pp.

\bibitem{mr85}  A. Medina, P. Revoy, Alg\`{e}bres de Lie et produit scalaire invariant (French) [Lie algebras and invariant scalar products], Ann. Sci. \'{E}cole Norm. Sup. (4) 18 (3)
    (1985), 553--561.



\bibitem{ovando16} G.P. Ovando, Lie algebras with ad-invariant metrics: A survey-guide, Rend. Semin. Mat. Univ. Politec. Torino  74 (1) (2016),  243--268.

\bibitem{pv91} A.M. Pastore,  F. Verroca, Some results on the homogeneous Riemannian structures of class $\mathcal{T}_{1}\oplus \mathcal{T}_{2}$, Rend. Mat. Appl. (7) 11 (1) (1991), 105--121.

\bibitem{tz21} A. Tohidfar,  A. Zaeim,  On pseudo-Riemannian cyclic homogeneous manifolds of dimension four,
J. Lie Theory 31 (1) (2021), 169--187.

\bibitem{tv83} F. Tricerri,  L. Vanhecke, Homogeneous structures on Riemannian manifolds, London Mathematical Society Lecture Note Series 83, Cambridge University Press, Cambridge, 1983.

\bibitem{zusman94} P. Zusmanovich, The second homology group of current Lie algebras, Ast\'{e}risque 226 (1994), 435--452.
























































\end{thebibliography}
\end{document}